\documentclass[12pt,a4]{amsart}
\usepackage[a4paper, left=28mm, right=28mm, top=28mm, bottom=34mm]{geometry}
\usepackage[dvips]{graphicx, color}
\usepackage{amsmath}
\usepackage{amscd}
\usepackage{amssymb}
\usepackage{amsthm}
\usepackage{mathrsfs}
\numberwithin{equation}{section}
\theoremstyle{definition}
\newtheorem{thm}{Theorem}[section]
\newtheorem{lem}[thm]{Lemma}
\newtheorem{cor}[thm]{Corollary}
\newtheorem{prop}[thm]{Proposition}

\theoremstyle{definition}

\newtheorem{rem}[thm]{Remark}

\newtheorem{defn}[thm]{Definition}

\newtheorem{ex}[thm]{Example}

\def\R{{\mathbb R}}
\def\Z{{\mathbb Z}}
\def\C{{\mathbb C}}
\def\O{{\mathscr O}}
\def\P{{\mathbb P}}

\DeclareMathOperator{\Homsheaf}{\mathscr{H}\!\textit{om}}
\DeclareMathOperator{\RHomsheaf}{R\mathscr{H}\!\textit{om}}
\DeclareMathOperator{\Extsheaf}{\mathscr{E}\!\textit{xt}}

\def\disc{\mathop{\mathrm{disc}}\nolimits}

\def\Hom{\mathop{\mathrm{Hom}}\nolimits}

\def\Ker{\mathop{\mathrm{Ker}}\nolimits}
\def\Norm{\mathop{\mathrm{N}}\nolimits}

\def\length{\mathop{\mathrm{length}}\nolimits}

\def\GL{\mathop{\mathrm{GL}}\nolimits}

\def\SL{\mathop{\mathrm{SL}}\nolimits}

\def\Res{\mathop{\mathrm{Res}}\nolimits}
\def\Sym{\mathop{\mathrm{Sym}}\nolimits}

\def\Proj{\mathop{\rm Proj}}

\def\Spec{\mathop{\rm Spec}}

\def\Supp{\mathop{\rm Supp}}

\def\rank{\mathop{\text{\rm rank}}\nolimits}

\def\chara{\mathop{\mathrm{char}}}

\newcommand{\transp}[1]{{}^{t}\!{#1}}

\begin{document}
\title[Sheaf theoretic classifications of pairs of square matrices]
{Sheaf theoretic  classifications of pairs of square matrices over arbitrary fields}
\author{Yasuhiro Ishitsuka}
\address{Department of Mathematics, Faculty of Science, Kyoto University, Kyoto 606-8502, Japan}
\email{yasu-ishi@math.kyoto-u.ac.jp}
\author{Tetsushi Ito}
\address{Department of Mathematics, Faculty of Science, Kyoto University, Kyoto 606-8502, Japan}
\email{tetsushi@math.kyoto-u.ac.jp}

\date{\today}


\maketitle

\begin{abstract}
We give classifications of linear orbits of pairs of square matrices
with non-vanishing discriminant polynomials over a field
in terms of certain coherent sheaves with additional data on closed subschemes of the projective line.
Our results are valid in a uniform manner over arbitrary fields including
those of characteristic two.
This work is based on the previous work of the first author on
theta characteristics on hypersurfaces.
As an application,
we give parametrizations of orbits of pairs of symmetric matrices
under special linear groups with fixed discriminant polynomials
generalizing some results of Wood and Bhargava-Gross-Wang.
\end{abstract}


\section{Introduction}
\label{Section:Introduction}

Classifying pencils of quadrics (or pairs of symmetric matrices)
is a classical topic in algebraic geometry and invariant theory,
which goes back at least to the work of Sylvester and Weierstrass
(\cite[Chapter XIII]{HodgePedoe}, \cite[Chapter 8, Historical Notes]{Dolgachev}).
There are vast amounts of literatures on the classification of pencils of quadrics
especially over $\R$ or $\C$ (e.g.\ See \cite{Thompson}, \cite{LeepSchueller} and references therein).
These days, thanks to the recent developments of Arithmetic Invariant Theory,
we become more interested in the classification of
pencils of quadrics over fields which are not necessarily algebraically closed
(such as global, local or finite fields)
because of its relation to the arithmetic of hyperelliptic curves
(\cite{IshitsukaGalois}, \cite{Wang}, \cite{BhargavaGrossWang1}, \cite{BhargavaGrossWang2}).

In this paper,
we give classifications of linear orbits of pairs of square matrices over a field
in terms of certain coherent sheaves with additional data on closed subschemes of
the projective line $\P^1$.
Our results are valid in a uniform manner over arbitrary fields including
those of characteristic two.
The methods used in this paper are similar to those in \cite{BeauvilleDeterminantal}, \cite{Ishitsuka}.
In \cite{Ishitsuka},
the first author obtained similar classifications
of linear orbits of $m$-tuples of symmetric matrices over arbitrary fields for any $m \geq 3$.
Coherent sheaves appearing in this paper are reminiscent of
{\it theta characteristics on hypersurfaces}.

Let us give the statement of our results on
the classification of linear orbits of pairs of symmetric matrices.
Geometrically, this case seems of fundamental interest because of its relation
to pencils of quadrics (see Section \ref{Section:SegreClassification}).
Let $k$ be a field, and $n \geq 1$ a positive integer.
The set of pairs of symmetric matrices of size $n+1$ with entries in $k$
is denoted by
\[ k^2 \otimes \Sym_2 k^{n+1} :=
   \big\{ M = (M_0,M_1) \ \big| \  M_i \in \mathrm{Mat}_{n+1}(k),\ \transp{M_i} = M_i \ (i=0,1) \big\}. \]
For a pair
$M = (M_0,M_1) \ \in \ k^2 \otimes \Sym_2 k^{n+1}$,
the {\it discriminant polynomial} is defined by
\[ \disc(M) := \det(X_0 M_0 + X_1 M_1). \]
Assume that $\disc(M) \neq 0$.
Then, it is a homogeneous polynomial of degree $n+1$ in two variables $X_0,X_1$.
The closed subscheme
\[ D_M := (\disc(M) = 0) \subset \P^1 \]
is called the {\it discriminant subscheme},
which is a zero-dimensional scheme over $k$ with $\dim_k H^0(D_M,\O_{D_M}) = n+1$.
There is a right action of
$P \in \GL_{n+1}(k)$ on $k^2 \otimes \Sym_2 k^{n+1}$ by
\[ M \cdot P := (\, \transp{P} M_0 P,\, \transp{P} M_1 P\,). \]
Since $\disc(M \cdot P) = (\det P)^2 \disc(M)$,
the action of $\GL_{n+1}(k)$ (resp.\ $\SL_{n+1}(k)$)
preserves the discriminant subschemes
(resp.\ discriminant polynomials).
Fix a closed subscheme $\iota \colon S \hookrightarrow \P^1$,
which is finite over $k$ and satisfies $\dim_k H^0(S,\O_S) = n+1$.
Our aim is to classify the set of
$\GL_{n+1}(k)$-orbits in $k^2 \otimes \Sym_2 k^{n+1}$
whose discriminant subschemes are $S$.
Let
\[ (k^2 \otimes \Sym_2 k^{n+1})_S \subset k^2 \otimes \Sym_2 k^{n+1} \]
be the subset consisting of pairs whose discriminant subschemes are $S$.
By Grothendieck duality (\cite{Hartshorne:ResiduesDuality}),
we have a coherent $\O_S$-module $(\iota^{!}\O_{\P^1}(-1))[1]$.
In fact, $(\iota^{!}\O_{\P^1}(-1))[1]$ is defined as a complex of $\O_S$-modules
in the derived category.
It is actually a free $\O_S$-module of rank one
concentrated in degree zero (see Lemma \ref{Lemma:Duality:Sheaf}).
We {\it fix} an isomorphism
\[
c \colon (\iota^{!}\O_{\P^1}(-1))[1] \overset{\sim}{\longrightarrow} \O_S
\]
of $\O_S$-modules.
(A canonical choice of $c$ does not seem to exist.
In general, bijections between linear orbits constructed in this paper depend on the choice of $c$
unless $k$ is algebraically closed.
See Remark \ref{Remark:Proposition:Key}.)
For a coherent $\O_S$-module $\mathcal{F}$,
its $\O_S$-linear dual
\[ \mathcal{F}^{\vee} := \Homsheaf_{S}\big( \mathcal{F},\,\O_S \big) \]
is isomorphic to $\mathcal{F}$,
but there is no canonical isomorphism between them
(see Lemma \ref{Lemma:Duality:SelfDual}).

The following is one of the main results of this paper.

\begin{thm}
\label{Theorem:MainTheorem}
There is a bijection between the following sets.
\begin{itemize}
\item The set of $\GL_{n+1}(k)$-orbits in $(k^2 \otimes \Sym_2 k^{n+1})_S$.
\item The set of equivalence classes of pairs $(\mathcal{F},\lambda)$, where
\begin{itemize}
\item $\mathcal{F}$ is a coherent $\O_S$-module with
$\length_{\O_{S,x}} \O_{S,x} = \length_{\O_{S,x}} {\mathcal{F}}_{x}$
for each $x \in S$, and
\item $\lambda \colon \mathcal{F} \overset{\sim}{\longrightarrow} \mathcal{F}^{\vee}$
is a symmetric isomorphism of $\O_{S}$-modules.
(For the definition of symmetric morphisms, see Definition \ref{Definition:TransposeSymmetric}.)
\end{itemize}
Here two pairs $(\mathcal{F},\lambda),(\mathcal{F}',\lambda')$
are equivalent if there is an isomorphism
$\rho \colon \mathcal{F} \overset{\sim}{\longrightarrow} \mathcal{F}'$
of $\O_{S}$-modules with $\lambda = (\transp{\rho}) \circ \lambda' \circ \rho$.
(For the definition of the transpose morphism $\transp{\rho}$,
see Definition \ref{Definition:TransposeSymmetric}.)
\end{itemize}
\end{thm}

For the classification of $\SL_{n+1}(k)$-orbits,
see Theorem \ref{Theorem:MainTheoremSL(n+1)}.
We also prove classification results for 
$\GL_{n+1}(k)$-orbits and $\SL_{n+1}(k)$-orbits for
pairs of square matrices which are not necessarily symmetric.
See Theorem \ref{Theorem:MainTheorem:NonSymmetric}
and Theorem \ref{Theorem:MainTheoremSL(n+1):NonSymmetric}.

Theorem \ref{Theorem:MainTheorem} may look slightly
different from other results in the literature because we use
the language of modern algebraic geometry.
Among a huge number of results in the literature,
the results in \cite[\S 3]{Milnor}, \cite[\S 4]{Waterhouse1}, \cite[\S 2]{Waterhouse2} seem closest to
Theorem \ref{Theorem:MainTheorem}.
(See also \cite{Williamson1}, \cite{Williamson2}, \cite{Venkatachaliengar}.)
Results like
Theorem \ref{Theorem:MainTheoremSL(n+1)},
Theorem \ref{Theorem:MainTheorem:NonSymmetric},
Theorem \ref{Theorem:MainTheoremSL(n+1):NonSymmetric}
were not considered in these papers.
Of course, when $k$ is algebraically closed of characteristic different from two,
the arguments in this paper are not very different from those in the literature.
One of the features of our methods is that
we systematically use the language and the results of modern algebraic geometry such
as Grothendieck duality and resolutions of coherent $\O_{\P^1}$-modules.
It makes, we believe, the statements and the proofs in this paper more streamlined.

The outline of this paper is as follows.
In Section \ref{Section:GrothendieckDualityResolution},
we recall basic results on Grothendieck duality and resolutions of coherent $\O_{\P^1}$-modules.
In Section \ref{Section:KeyBijection}, following the strategy in \cite{Ishitsuka},
we first prove a rigidified bijection (Proposition \ref{Proposition:Key}).
Then, it is straightforward to deduce Theorem \ref{Theorem:MainTheorem} from it.
We need to modify some of the arguments in \cite{Ishitsuka} because some of the key results
such as \cite[Proposition 2.6]{Ishitsuka} (or \cite[Proposition 1.11]{BeauvilleDeterminantal})
are not true over $\P^1$.
In Section \ref{Section:WoodBhargavaGrossWang},
we give parametrizations of $\SL_{n+1}(k)$-orbits of pairs of symmetric matrices
with fixed discriminant polynomials generalizing some results of Wood and Bhargava-Gross-Wang
(\cite{Wood}, \cite{BhargavaGrossWang1}, \cite{BhargavaGrossWang2}).
In fact, generalizing their results using the methods of modern algebraic geometry
was one of the initial motivation of this work.
In Section \ref{Section:SegreClassification},
we explain how to recover the classical classification of pencils of quadrics
in terms of Segre symbols from our results.
In Section \ref{Section:ClassificationPairs:NonSymmetric},
we give a classification of pairs of matrices which are not necessarily symmetric.
Finally, in Appendix \ref{Appendix:SymmetricDeterminant},
we give calculations of the determinants of certain symmetric matrices.

\begin{rem}
Similar classification results for pairs of square matrices over commutative rings are
obtained by Wood (\cite{Wood}).
{\it Balanced pairs of modules} in \cite[Section 6]{Wood}
are closely related to pairs $(\mathcal{F},\lambda)$ in Theorem \ref{Theorem:MainTheorem}
(or $(\mathcal{F},\psi)$ in Theorem \ref{Theorem:MainTheorem:NonSymmetric}).
It is an interesting problem to study the relation between our results and the results of Wood.
The statements of the results are similar,
but the proofs are different.
Note that, in this paper, we make a non-canonical choice of $c$,
whereas such a choice does not seem to appear in \cite{Wood}.
It seems that the choice of $c$ is somehow incorporated in Wood's construction in a subtle way.
\end{rem}

\begin{rem}
It is an interesting problem to generalize our methods
to classify pairs of square matrices with {\it vanishing} discriminant polynomials.
To this direction, classical results due to Kronecker and Dickson are
well-known (\cite[\S 3]{Waterhouse1}).
Another interesting problem is to generalize our methods
to classify pencils of quadrics in characteristic two (\cite{Bhosle}).
\end{rem}

\subsection*{Acknowledgements}

The work of the first author was supported by JSPS KAKENHI Grant Number 13J01450.
The work of the second author 
was supported by JSPS KAKENHI Grant Number 20674001 and 26800013.

\section{Grothendieck duality and resolutions of coherent $\O_{\P^1}$-modules}
\label{Section:GrothendieckDualityResolution}

We fix a field $k$, and a positive integer $n \geq 1$.
The base field $k$ is arbitrary unless otherwise mentioned.
We also fix a closed subscheme $\iota \colon S \hookrightarrow \P^1$,
which is finite over $k$ and satisfies $\dim_k H^0(S,\O_S) = n+1$.

We recall some results on Grothendieck duality for the closed immersion $\iota$
(\cite{Hartshorne:ResiduesDuality}).
There is a functor $\iota^{!}$ from the bounded derived category of coherent $\O_{\P^1}$-modules
to the bounded derived category of coherent $\O_S$-modules.
For a bounded complex $\mathcal{F}$ (resp.\ $\mathcal{G}$) of coherent
$\O_S$-modules (resp.\ $\O_{\P^1}$-modules),
there is a canonical quasi-isomorphism
\[ \iota_{\ast} \RHomsheaf_{S}(\mathcal{F},\,\iota^{!} \mathcal{G})
   \cong \RHomsheaf_{\P^1}(\iota_{\ast} \mathcal{F},\,\mathcal{G}). \]

\begin{lem}
\label{Lemma:Duality:QuasiIsom}
Let $\mathcal{F}$ be a coherent $\O_S$-module,
and $\mathcal{F}'$ a free $\O_S$-module of rank one.
Then we have $\Extsheaf^q_{S}( \mathcal{F},\, \mathcal{F}') = 0$ for any $q \geq 1$,
and there is a canonical quasi-isomorphism
\[ \Homsheaf_{S}( \mathcal{F},\, \mathcal{F}') \cong \RHomsheaf_{S}( \mathcal{F},\, \mathcal{F}'). \]
\end{lem}

\begin{proof}
Since the closed subscheme $S \subset \P^1$ is a complete intersection,
$S$ is a zero-dimensional Gorenstein scheme.
Hence $\mathcal{F}'$ is injective as an $\O_S$-module (\cite[Proposition 21.5]{Eisenbud}),
and we have $\Extsheaf^q_{S}( \mathcal{F},\, \mathcal{F}') = 0$ for any $q \geq 1$.
The second assertion follows from this because
the $q$-th cohomology sheaf of $\RHomsheaf_{S}( \mathcal{F},\, \mathcal{F}')$
is $\Extsheaf^q_{S}( \mathcal{F},\, \mathcal{F}')$.
\end{proof}

\begin{lem}
\label{Lemma:Duality:Sheaf}
For a locally free $\O_{\P^1}$-module $\mathcal{G}$ of rank one,
the complex $\iota^{!} \mathcal{G}$ is a free $\O_S$-module of rank one
shifted by one to the right.
(In other words, $(\iota^{!} \mathcal{G})[1]$ is quasi-isomorphic to
a free $\O_S$-module of rank one concentrated in degree zero.
Here ``$[1]$'' is the shift functor in the derived category
(\cite{Hartshorne:ResiduesDuality}).)
\end{lem}

\begin{proof}
Since $S$ is zero-dimensional,
any locally free $\O_{\P^1}$-module $\mathcal{G}$ of rank one
is isomorphic to the canonical sheaf $\Omega^1_{\P^1}$ on some open neighborhood of $S \subset \P^1$.
Hence it is enough to prove the assertion when $\mathcal{G} = \Omega^1_{\P^1}$.
Since $\Omega^1_{\P^1}[1]$ is quasi-isomorphic to the dualizing complex on $\P^1$
(\cite[Chapter VII, Theorem 4.1]{Hartshorne:ResiduesDuality}),
$(\iota^{!} \Omega^1_{\P^1})[1]$ is quasi-isomorphic to the dualizing complex on $S$.
Hence $(\iota^{!} \Omega^1_{\P^1})[1]$ is a free $\O_S$-module of rank one
concentrated in degree zero because $S$ is a zero-dimensional Gorenstein scheme.
(See the proof of Lemma \ref{Lemma:Duality:QuasiIsom}.
See also \cite[Chapter V, Proposition 9.3]{Hartshorne:ResiduesDuality}.)
\end{proof}

\begin{lem}
\label{Lemma:Duality:Ext}
Let $\mathcal{F}$ be a coherent $\O_S$-module,
and $\mathcal{G}$ a locally free $\O_{\P^1}$-module of rank one.
Then there is a canonical isomorphism
\[ \iota_{\ast} \Homsheaf_{S}\big( \mathcal{F},\, (\iota^{!} \mathcal{G})[1] \big)
     \cong \Extsheaf^1_{\P^1}( \iota_{\ast} \mathcal{F},\, \mathcal{G}). \]
\end{lem}

\begin{proof}
By Lemma \ref{Lemma:Duality:Sheaf}, Lemma \ref{Lemma:Duality:QuasiIsom}
and Grothendieck duality, we have a canonical quasi-isomorphism
$\iota_{\ast} \Homsheaf_{S}\big( \mathcal{F},\, (\iota^{!} \mathcal{G})[1] \big)
 \cong \RHomsheaf_{\P^1}(\iota_{\ast} \mathcal{F},\,\mathcal{G}[1])$.
The cohomology sheaves of
$\RHomsheaf_{\P^1}(\iota_{\ast} \mathcal{F},\,\mathcal{G}[1])$
are concentrated in degree zero.
Hence we have
$\RHomsheaf_{\P^1}(\iota_{\ast} \mathcal{F},\,\mathcal{G}[1])
\cong \Extsheaf^1_{\P^1}( \iota_{\ast} \mathcal{F},\, \mathcal{G})$.
\end{proof}

\begin{defn}
\label{Definition:TransposeSymmetric}
Let $\mathcal{F}, \mathcal{F}', \mathcal{G}$ be coherent $\O_S$-modules.

\begin{enumerate}
\item For a morphism $\rho \colon \mathcal{F} \longrightarrow \mathcal{F}'$ of $\O_S$-modules,
the canonical morphism
\[
\Homsheaf_{S}( \mathcal{F}',\, \mathcal{G}) \longrightarrow
\Homsheaf_{S}( \mathcal{F},\, \mathcal{G}),\qquad
\varphi \mapsto (x \mapsto \varphi(\rho(x)))
\]
is called the {\it transpose} of $\rho$. We denote it by $\transp{\rho}$.

\item For a morphism
$\lambda \colon \mathcal{F} \longrightarrow \Homsheaf_{S}( \mathcal{F},\, \mathcal{G})$
of $\O_S$-modules, by composing the transpose $\transp{\lambda}$
with the canonical morphism
\[
\mathrm{can} \colon \mathcal{F} \longrightarrow 
\Homsheaf_{S}\big( \Homsheaf_{S}( \mathcal{F},\, \mathcal{G}),\, \mathcal{G}\, \big),
\qquad x \mapsto (\psi \mapsto \psi(x)),
\]
we obtain $\transp{\lambda} \circ \mathrm{can} \colon
\mathcal{F} \longrightarrow \Homsheaf_{S}( \mathcal{F},\, \mathcal{G})$.
We say $\lambda$ is {\it symmetric} if $\lambda = \transp{\lambda} \circ \mathrm{can}$.
\end{enumerate}
\end{defn}

\begin{lem}
\label{Lemma:RankOneSymmetric}
Let $\mathcal{F}, \mathcal{G}$ be coherent $\O_S$-modules.
If $\mathcal{F}$ is a free $\O_S$-module of rank one,
any morphism
$\lambda \colon \mathcal{F} \longrightarrow \Homsheaf_{S}( \mathcal{F},\, \mathcal{G})$
of $\O_S$-modules is symmetric.
\end{lem}

\begin{proof}
We may assume $\mathcal{F} = \O_S$.
We put $A := H^0(S, \O_S)$.
Then, $N := H^0(S, \mathcal{G})$ is a finitely generated $A$-module.
The morphism
$\lambda \colon \mathcal{F} \longrightarrow \Homsheaf_{S}( \mathcal{F},\, \mathcal{G})$
corresponds to a homomorphism
$\lambda \colon A \longrightarrow \Hom_{A}(A,N)$ of $A$-modules.
(We denote it by the same letter $\lambda$.)
The composite of
\[ \transp{\lambda} \circ \mathrm{can} \colon
   A \longrightarrow \Hom_{A}\big( \Hom_{A}(A,N), N \big)
   \longrightarrow \Hom_{A}(A,N) \]
is $a \mapsto (\varphi \mapsto \varphi(a)) \mapsto (b \mapsto (\lambda(b))(a))$
for any $a,b \in A$.
Since 
$(\lambda(a))(b)  = a b (\lambda(1))(1) = (\lambda(b))(a)$,
we see that $\lambda$ is symmetric.
\end{proof}

\begin{lem}
\label{Lemma:Duality:SelfDual}
For a coherent $\O_S$-module $\mathcal{F}$,
there is a symmetric isomorphism
$\lambda \colon \mathcal{F} \overset{\sim}{\longrightarrow}
\mathcal{F}^{\vee} := \Homsheaf_{S}\big( \mathcal{F},\,\O_S \big)$.
\end{lem}

\begin{proof}
Since $S$ is zero-dimensional,
we write the closed points of $S$ as $S = \{ x_0,\ldots,x_r \}$.
We denote the coordinate ring of $x_i$ as $A_i := H^0(x_i,\O_{x_i})$.
Then, $S$ is the disjoint union of $\Spec A_i$,
and each $A_i$ is a $k$-algebra isomorphic to $k_i[T]/(T^{m_i})$
for some finite extension $k_i/k$ and some $m_i \geq 1$.
Hence we may assume $S = \Spec k'[T]/(T^m)$
for a finite extension $k'/k$ and $m \geq 1$.
We may assume $M = k'[T]/(T^{m'})$ for some $m' \leq m$.
Then, by \cite[Chapter VII, \S 4.9]{BourbakiAlgebraII},
there is an isomorphism
$M \cong \Hom_{k'[T]/(T^m)}\big( M,\,k'[T]/(T^m) \big)$ of $k'[T]/(T^m)$-modules.
Any such isomorphism is symmetric by
Lemma \ref{Lemma:RankOneSymmetric}.
\end{proof}

\begin{lem}
\label{Lemma:DiagonalizationQuadraticForm}
Assume that $\chara k \neq 2$.
Let $\mathcal{F}$ be a coherent $\O_S$-module, and $\mathcal{G}$ a free $\O_S$-module of rank one.
Let $\lambda \colon \mathcal{F} \overset{\sim}{\longrightarrow} \Homsheaf_{S}( \mathcal{F},\, \mathcal{G})$
be a symmetric isomorphism.
Then the pair $(\mathcal{F}, \lambda)$ can be decomposed as
$(\mathcal{F}, \lambda) \cong \big( \bigoplus_{i=0}^r \mathcal{F}_i ,\, \bigoplus_{i=0}^r \lambda_i \big)$,
where, for each $i$,
$\mathcal{F}_i$ is a coherent $\O_S$-module and
$\lambda_i \colon \mathcal{F}_i \overset{\sim}{\longrightarrow} \Homsheaf_{S}( \mathcal{F}_i,\, \mathcal{G})$
is a symmetric isomorphism
such that $\Supp \mathcal{F}_i$ consists of one point
and $H^0(S,\mathcal{F}_i)$ is generated by one element
as an $H^0(S,\O_S)$-module.
\end{lem}

\begin{proof}
We may assume $S = \Spec k'[T]/(T^m)$
for a finite extension $k'/k$ and $m \geq 1$
(see the proof of Lemma \ref{Lemma:Duality:SelfDual}).
Because $2$ is invertible in $k'[T]/(T^m)$,
the assertion follows from the fact that
every symmetric inner product space over $k'[T]/(T^m)$ has an orthogonal basis
(\cite[Chapter 1, Corollary 3.4]{MilnorHusemoller}, \cite[\S 1, Theorem 1.1]{Waterhouse1}).
\end{proof}

\begin{rem}
The isomorphism classes of $\mathcal{F}_0,\ldots,\mathcal{F}_r$
are, up to permutation, independent of the choice of $\lambda$.
They are uniquely determined by $\mathcal{F}$ by
\cite[Chapter VII, \S 4.9]{BourbakiAlgebraII}.
\end{rem}

\begin{lem}
\label{Lemma:ClassificationQuadraticForm}
Assume that $k$ is algebraically closed of characteristic different from two.
Let $\mathcal{F}$ be a coherent $\O_S$-module,
and $\mathcal{G}$ a free $\O_S$-module of rank one.
Let $\lambda,\lambda' \colon \mathcal{F} \overset{\sim}{\longrightarrow} \Homsheaf_{S}( \mathcal{F},\, \mathcal{G})$
be symmetric isomorphisms.
Then there is an automorphism
$\rho \colon \mathcal{F} \overset{\sim}{\longrightarrow} \mathcal{F}$
with $\lambda = (\transp{\rho}) \circ \lambda' \circ \rho$.
\end{lem}

\begin{proof}
We may assume $S = \Spec k'[T]/(T^m)$
for a finite extension $k'/k$ and $m \geq 1$
(see the proof of Lemma \ref{Lemma:Duality:SelfDual}).
Then, $M := H^0(S,\mathcal{F})$ is a finitely generated $k'[T]/(T^m)$-module.
By Lemma \ref{Lemma:DiagonalizationQuadraticForm},
the pair $(M,\lambda)$ can be decomposed as
$(M,\lambda) \cong \big( \bigoplus_{i=0}^r k'[T]/(T^{m_i}),\,\lambda_i \big)$.
The integers $m_0,\ldots,m_r$ are uniquely determined by $M$
(\cite[Chapter VII, \S 4.9]{BourbakiAlgebraII}).
The same is true for $\lambda'$.
Hence we may assume $M = k'[T]/(T^{m_i})$ for some $m_i \leq m$.
In this case, the assertion follows from the fact that
$(k'[T]/(T^{m_i}))^{\times} \longrightarrow (k'[T]/(T^{m_i}))^{\times},\ \alpha \mapsto \alpha^2$
is surjective.
(It is a consequence of Hensel's lemma because
$k$ is algebraically closed of characteristic different from two.)
\end{proof}

\begin{lem}
\label{Lemma:Resolution}
Let $\mathcal{F}$ be a coherent $\O_S$-module with
$\dim_k H^0(S,\mathcal{F}) = r+1$.
Let $\{ s_0,s_1,\ldots,s_r \}$ be an ordered $k$-basis of
$H^0(S,\mathcal{F})$.
Let $p \colon \O_{\P^1}^{\oplus (r+1)} \longrightarrow \iota_{\ast} \mathcal{F}$
be a surjective morphism of $\O_{\P^1}$-modules
which sends
the standard $k$-basis of $H^0(\P^1,\O_{\P^1}^{\oplus (r+1)}) = k^{\oplus (r+1)}$
to $\{ s_0,s_1,\ldots,s_r \}$.
Then the kernel $\Ker p$ is a locally free $\O_{\P^1}$-module
isomorphic to $\O_{\P^1}(-1)^{\oplus (r+1)}$.
\end{lem}

\begin{proof}
This lemma is well-known at least when $k$ is algebraically closed.
The following proof works over arbitrary fields.
From the short exact sequence
\[
\begin{CD}
0 @>>> \Ker p
  @>>> \O_{\P^1}^{\oplus (r+1)}
  @>{p}>> \iota_{\ast} \mathcal{F} @>>> 0,
\end{CD}
\]
we see that $\Ker p$ is torsion-free.
It is locally free because $\P^1$ is regular and one-dimensional.
Hence we have $\Ker p \cong \bigoplus_{i=0}^r \O_{\P^1}(a_i)$
for some $a_i \in \Z$.
(This fact is sometimes referred to as
{\it Birkhoff-Grothendieck's classification of vector bundles on $\P^1$}.
For a proof of it which works over arbitrary fields,
see \cite{HazewinkelMartin}, \cite[Theorem 1.3.1]{HuybrechtsLehn}.)
We have $H^0(\P^1,\O_{\P^1}(a_i)) = H^1(\P^1,\O_{\P^1}(a_i)) = 0$
for all $i$ because $H^1(\P^1,\O_{\P^1}) = 0$.
Hence we have $a_i = -1$ for all $i$ by
\cite[Chapter III, Theorem 5.1]{Hartshorne:AlgebraicGeometry}.
\end{proof}

\section{Bijections on pairs of symmetric matrices}
\label{Section:KeyBijection}

In this section, we first prove a rigidified bijection between pairs of symmetric matrices
and equivalence classes of certain coherent sheaves with additional data
on discriminant subschemes.
Then we prove Theorem \ref{Theorem:MainTheorem}
by considering the action of $\GL_{n+1}(k)$.
A variant of Theorem \ref{Theorem:MainTheorem} for
$\SL_{n+1}(k)$-orbits can be proved by the same method.

As in Section \ref{Section:GrothendieckDualityResolution},
we fix a field $k$, a positive integer $n \geq 1$,
and a closed subscheme $\iota \colon S \hookrightarrow \P^1$
which is finite over $k$ and satisfies $\dim_k H^0(S,\O_S) = n+1$.
By Lemma \ref{Lemma:Duality:Sheaf},
we see that $(\iota^{!}\O_{\P^1}(-1))[1]$ is a free $\O_S$-module of rank one.
We {\it fix} an isomorphism
\[ c \colon (\iota^{!}\O_{\P^1}(-1))[1] \overset{\sim}{\longrightarrow} \O_S \]
of $\O_S$-modules.

\begin{prop}
\label{Proposition:Key}
\begin{enumerate}
\item There is a bijection between the following sets.
\begin{itemize}
\item The set $(k^2 \otimes \Sym_2 k^{n+1})_S$.
\item The set of equivalence classes of triples $(\mathcal{F},\lambda,s)$,
where
\begin{itemize}
\item $\mathcal{F}$ is a coherent $\O_{S}$-module with
$\length_{\O_{S,x}} \O_{S,x} = \length_{\O_{S,x}} {\mathcal{F}}_{x}$
for each $x \in S$,
\item $\lambda \colon \mathcal{F} \overset{\sim}{\longrightarrow} \mathcal{F}^{\vee}$
is a symmetric isomorphism of $\O_{S}$-modules, and
\item $s = \{ s_0,s_1,\ldots,s_n \}$ is an ordered $k$-basis of $H^0(S,\mathcal{F})$.
\end{itemize}
Here two triples $(\mathcal{F},\lambda,s),(\mathcal{F}',\lambda',s')$
are equivalent if there is an isomorphism
$\rho \colon \mathcal{F} \overset{\sim}{\longrightarrow} \mathcal{F}'$
of $\O_{S}$-modules with $\lambda = (\transp{\rho}) \circ \lambda' \circ \rho$
and $\rho(s)  = s'$.
\end{itemize}

\item For a pair $M = (M_0,M_1) \in (k^2 \otimes \Sym_2 k^{n+1})_S$,
take a triple $(\mathcal{F},\lambda,s)$ corresponding to $M$ by (1).
For an invertible matrix $P \in \GL_{n+1}(k)$, define the ordered $k$-basis
$s\,\transp{P}^{-1} = \{ (s\,\transp{P}^{-1})_i \}_{0 \leq i \leq n}$ by
\[ (s\,\transp{P}^{-1})_i := \sum_{j=0}^{n} (\transp{P}^{-1})_{j,i} \, s_j, \]
where $(\transp{P}^{-1})_{j,i} \in k$ is the $(j,i)$-entry of $\transp{P}^{-1}$.
Then, the pair $M \cdot P := (\, \transp{P} M_0 P,\, \transp{P} M_1 P\,)$
corresponds to the triple $(\mathcal{F},\lambda, s\,\transp{P}^{-1})$ by (1).
\end{enumerate}
\end{prop}

\begin{rem}
\label{Remark:Proposition:Key}
The bijection in Proposition \ref{Proposition:Key} (1)
depends on the choice of $c$.
In order to make the statement independent of such choices,
we need to replace $\mathcal{F}^{\vee}$
by $\Homsheaf_{S}\big( \mathcal{F},\,(\iota^{!}\O_{\P^1}(-1))[1] \big)$.
When $k$ is algebraically closed of characteristic different from two,
it is easy to see that
the bijections in our classification theorems for $\GL_{n+1}(k)$-orbits
(Theorem \ref{Theorem:MainTheorem} and Theorem \ref{Theorem:MainTheorem:NonSymmetric})
are independent of the choice of $c$.
\end{rem}

\begin{proof}
(1) \ For a pair $M = (M_0,M_1) \in (k^2 \otimes \Sym_2 k^{n+1})_S$,
multiplying the matrix $X_0 M_0 + X_1 M_1$ on column vectors,
we get a morphism
\[ M \colon \O_{\P^1}(-1)^{\oplus (n+1)} \longrightarrow \O_{\P^1}^{\oplus (n+1)} \]
of $\O_{\P^1}$-modules.
We define a coherent $\O_{\P^1}$-module $\mathcal{F}_0$
by the short exact sequence
\begin{equation}
\label{Proposition:Key:ShortExactSequence1}
\begin{CD}
0 @>>> \O_{\P^1}(-1)^{\oplus (n+1)}
  @>{M}>> \O_{\P^1}^{\oplus (n+1)}
  @>>> \mathcal{F}_0 @>>> 0.
\end{CD}
\end{equation}
For a section $\alpha$ of $\O_{\P^1}^{\oplus (n+1)}$,
we have $(\det M) \alpha = M \mathrm{adj}(M) \alpha$,
where $\mathrm{adj}(M)$ (resp.\ $\det M$) is
the adjugate matrix of $X_0 M_0 + X_1 M_1$ (resp.\ the determinant $\det(X_0 M_0 + X_1 M_1)$).
We see that $\mathcal{F}_0$ comes from an $\O_S$-module.
We put $\mathcal{F} := \iota^{\ast} \mathcal{F}_0$.
Then we have $\mathcal{F}_0 = \iota_{\ast} \mathcal{F}$.
For each $x \in S$, the length of ${\mathcal{F}}_{x}$ as an $\O_{S,x}$-module
is equal to the order of vanishing of the polynomial
$\det(X_0 M_0 + X_1 M_1)$ at $x$.
Hence we have $\length_{\O_{S,x}} \O_{S,x} = \length_{\O_{S,x}} {\mathcal{F}}_{x}$.
Since $H^0(\P^1,\O_{\P^1}(-1)) = H^0(\P^1,\O_{\P^1}(-1)) = 0$, we have an isomorphism
\[ H^0(\P^1,\O_{\P^1}^{\oplus (n+1)}) = k^{\oplus (n+1)} \overset{\sim}{\longrightarrow} H^0(S,\mathcal{F}) \]
of $k$-vector spaces.
We denote the image of the standard $k$-basis by $s = \{ s_0,s_1,\ldots,s_n \}$.
Applying the left derived functor of
$\mathcal{G} \mapsto \Homsheaf_{\P^1}(\mathcal{G},\O_{\P^1}(-1))$ to
(\ref{Proposition:Key:ShortExactSequence1}),
we have
\begin{equation}
\label{Proposition:Key:ShortExactSequence2}
\begin{CD}
0 @>>> \O_{\P^1}(-1)^{\oplus (n+1)}
  @>{\transp{M}}>> \O_{\P^1}^{\oplus (n+1)}
  @>>> \Extsheaf^1_{\P^1}(\iota_{\ast} \mathcal{F},\,\O_{\P^1}(-1))
 @>>> 0
\end{CD}
\end{equation}
because
$\Homsheaf_{\P^1}(\iota_{\ast} \mathcal{F},\O_{\P^1}(-1))
= \Extsheaf^1_{\P^1} \big( \O_{\P^1}^{\oplus (n+1)},\,\O_{\P^1}(-1) \big) = 0$
(\cite[Chapter III, Proposition 6.3, Proposition 6.7]{Hartshorne:AlgebraicGeometry}).
By Lemma \ref{Lemma:Duality:Ext} and using the isomorphism $c$, we have
\[ \Extsheaf^1_{\P^1}( \iota_{\ast} \mathcal{F},\, \O_{\P^1}(-1))
\cong \iota_{\ast} \Homsheaf_{S}\big( \mathcal{F},\, (\iota^{!} \O_{\P^1}(-1))[1] \big)
\overset{c}{\cong} \iota_{\ast}(\mathcal{F}^{\vee}). \]
Since $M$ is a pair of symmetric matrices, we identify
(\ref{Proposition:Key:ShortExactSequence1}), (\ref{Proposition:Key:ShortExactSequence2}),
and get an isomorphism
$\mathcal{F}_0 = \iota_{\ast} \mathcal{F} \cong \iota_{\ast}(\mathcal{F}^{\vee})$.
We denote this isomorphism by $\iota_{\ast}(\lambda)$ for a unique isomorphism
$\lambda \colon \mathcal{F} \overset{\sim}{\longrightarrow} \mathcal{F}^{\vee}$
of $\O_S$-modules. Applying the left derived functor of
$\mathcal{G} \mapsto \Homsheaf_{\P^1}(\mathcal{G},\O_{\P^1}(-1))$
to (\ref{Proposition:Key:ShortExactSequence2}),
we get
\begin{equation}
\label{Proposition:Key:ShortExactSequence3}
\begin{CD}
0 @>>> \O_{\P^1}(-1)^{\oplus (n+1)}
  @>{\transp(\transp{M}) = M}>> \O_{\P^1}^{\oplus (n+1)}
  @>>> \Extsheaf^1_{\P^1}(\iota_{\ast} (\mathcal{F}^{\vee}),\,\O_{\P^1}(-1))
 @>>> 0.
\end{CD}
\end{equation}
Summarizing
(\ref{Proposition:Key:ShortExactSequence1}), (\ref{Proposition:Key:ShortExactSequence2}),
(\ref{Proposition:Key:ShortExactSequence3}),
we have the following commutative diagram.
\[
\begin{CD}
0 @>>> \O_{\P^1}(-1)^{\oplus (n+1)}
  @>{M}>> \O_{\P^1}^{\oplus (n+1)}
  @>>> \mathcal{F}_0 = \iota_{\ast} \mathcal{F} @>>> 0 \\
 @. @VV{\mathrm{id}}V @VV{\mathrm{id}}V @VV{\iota_{\ast}(\lambda)}V \\ 
0 @>>> \O_{\P^1}(-1)^{\oplus (n+1)}
  @>{\transp{M}}>> \O_{\P^1}^{\oplus (n+1)}
  @>>> \Extsheaf^1_{\P^1}(\iota_{\ast} \mathcal{F},\,\O_{\P^1}(-1)) \cong \iota_{\ast} (\mathcal{F}^{\vee})
 @>>> 0 \\
 @. @AA{\mathrm{id}}A @AA{\mathrm{id}}A @AA{\iota_{\ast}(\transp{\lambda} \circ \mathrm{can})}A \\ 
0 @>>> \O_{\P^1}(-1)^{\oplus (n+1)}
  @>{\transp(\transp{M}) = M}>> \O_{\P^1}^{\oplus (n+1)}
  @>>> \Extsheaf^1_{\P^1}(\iota_{\ast} (\mathcal{F}^{\vee}),\,\O_{\P^1}(-1)) \cong \iota_{\ast} \mathcal{F}
 @>>> 0
\end{CD}
\]
By the commutativity of this diagram,
we see that $\lambda = \transp{\lambda} \circ \mathrm{can}$.
Hence $\lambda$ is symmetric.

Conversely, for a triple $(\mathcal{F},\lambda,s)$, there is a unique surjection
$p \colon \O_{\P^1}^{\oplus (n+1)} \longrightarrow \iota_{\ast} \mathcal{F}$
which sends the standard $k$-basis of  $H^0(\P^1,\O_{\P^1}^{\oplus (n+1)})$
to the ordered $k$-basis $s$ of $H^0(S,\mathcal{F})$.
The kernel $\Ker p$ is isomorphic to $\O_{\P^1}(-1)^{\oplus (n+1)}$
by Lemma \ref{Lemma:Resolution}.
Applying the left derived functor of
$\mathcal{G} \mapsto \Homsheaf_{\P^1}(\mathcal{G},\O_{\P^1}(-1))$ to
\begin{equation}
\label{Proposition:Key:ShortExactSequence4}
\begin{CD}
0 @>>> \Ker p
  @>>> \O_{\P^1}^{\oplus (n+1)}
  @>{p}>> \iota_{\ast} \mathcal{F} @>>> 0,
\end{CD}
\end{equation}
we get the following short exact sequence
\begin{equation}
\label{Proposition:Key:ShortExactSequence5}
\begin{CD}
0 @>>> \O_{\P^1}(-1)^{\oplus (n+1)}
  @>>> \Homsheaf_{\P^1}(\Ker p,\O_{\P^1}(-1))
  @>>> \iota_{\ast}(\mathcal{F}^{\vee})
 @>>> 0.
\end{CD}
\end{equation}
We combine (\ref{Proposition:Key:ShortExactSequence4}), (\ref{Proposition:Key:ShortExactSequence5})
using $\lambda$.
We get the following commutative diagram
\begin{equation}
\label{Proposition:Key:CommutativeDiagram1}
\begin{CD}
0 @>>> \Ker p
  @>>> \O_{\P^1}^{\oplus (n+1)}
  @>{p}>> \iota_{\ast} \mathcal{F} @>>> 0 \\
& & @VV{\cong, q}V @VV_{\cong}V @VV{\iota_{\ast}(\lambda)}V \\
0 @>>> \O_{\P^1}(-1)^{\oplus (n+1)}
  @>>> \Homsheaf_{\P^1}(\Ker p,\O_{\P^1}(-1))
  @>>> \iota_{\ast}(\mathcal{F}^{\vee})
 @>>> 0.
\end{CD}
\end{equation}
The isomorphism $q$ in (\ref{Proposition:Key:CommutativeDiagram1})
is {\it uniquely determined} by the commutativity of (\ref{Proposition:Key:CommutativeDiagram1}).
Using $q$, the short exact sequence (\ref{Proposition:Key:ShortExactSequence4}) becomes
\begin{equation}
\label{Proposition:Key:ShortExactSequence6}
\begin{CD}
0 @>>> \O_{\P^1}(-1)^{\oplus (n+1)}
  @>>> \O_{\P^1}^{\oplus (n+1)}
  @>{p}>> \iota_{\ast} \mathcal{F} @>>> 0.
\end{CD}
\end{equation}
The morphism 
$\O_{\P^1}(-1)^{\oplus (n+1)} \longrightarrow \O_{\P^1}^{\oplus (n+1)}$
in (\ref{Proposition:Key:ShortExactSequence6})
gives a pair $M$ of square matrices of size $n+1$.
It is straightforward to see that $M$ is a pair of symmetric matrices because
$\lambda$ is symmetric.
It is easy to see that,
if two triples $(\mathcal{F},\lambda,s)$, $(\mathcal{F}',\lambda',s')$
are equivalent, the symmetric matrices obtained from them are the same.

The two maps
$M \mapsto (\mathcal{F},\lambda,s)$ and $(\mathcal{F},\lambda,s) \mapsto M$
constructed above give a desired bijection.

\vspace{0.1in}

\noindent
(2) \ This follows from the commutativity of the following diagram
\[
\begin{CD}
0 @>>> \O_{\P^1}(-1)^{\oplus (n+1)}
  @>{M}>> \O_{\P^1}^{\oplus (n+1)}
  @>{p}>> \iota_{\ast} \mathcal{F} @>>> 0 \\
& & @AA{\cong, P}A @AA{\cong, \transp{P}^{-1}}A @AA{\mathrm{id}}A \\
0 @>>> \O_{\P^1}(-1)^{\oplus (n+1)}
  @>{M \cdot P}>> \O_{\P^1}^{\oplus (n+1)} @>{p'}>> \iota_{\ast} \mathcal{F} @>>> 0.
\end{CD}
\]
Here $p'$ sends the standard $k$-basis of
$H^0(\P^1,\O_{\P^1}^{\oplus (n+1)})$
to the ordered $k$-basis $s\,\transp{P}^{-1}$ of $H^0(S,\mathcal{F})$.
\end{proof}

It is straightforward to deduce Theorem \ref{Theorem:MainTheorem} from
Proposition \ref{Proposition:Key}.

\begin{proof}[{\bfseries\itshape Proof of Theorem \ref{Theorem:MainTheorem}}]
For pairs $M,M' \in (k^2 \otimes \Sym_2 k^{n+1})_S$,
let $(\mathcal{F},\lambda,s)$ (resp.\ $(\mathcal{F}',\lambda',s')$)
be a triple corresponding to $M$ (resp.\ $M'$)
by Proposition \ref{Proposition:Key}.
We shall show that
$M,M'$ are in the same $\GL_{n+1}(k)$-orbit if and only if
$(\mathcal{F},\lambda)$, $(\mathcal{F}',\lambda')$ are equivalent.

Assume that $M,M'$ are in the same $\GL_{n+1}(k)$-orbit.
We take $P \in \GL_{n+1}(k)$ with $M' = M \cdot P$.
By Proposition \ref{Proposition:Key},
there is $\rho \colon \mathcal{F} \overset{\sim}{\longrightarrow} \mathcal{F}'$
with $\lambda = (\transp{\rho}) \circ \lambda' \circ \rho$
and $\rho(s\,\transp{P}^{-1})  = s'$.
Hence $(\mathcal{F},\lambda)$, $(\mathcal{F}',\lambda')$ are equivalent.

Conversely, assume that $(\mathcal{F},\lambda)$, $(\mathcal{F}',\lambda')$
are equivalent. There is
$\rho \colon \mathcal{F} \overset{\sim}{\longrightarrow} \mathcal{F}'$
with $\lambda = (\transp{\rho}) \circ \lambda' \circ \rho$.
Since $\rho(s),s'$ are ordered $k$-bases of $H^0(S,\mathcal{F}')$,
there is $P \in \GL_{n+1}(k)$ with
$\rho(s\,\transp{P}^{-1})  = s'$.
Then $(\mathcal{F},\lambda,s\,\transp{P}^{-1})$, $(\mathcal{F}',\lambda',s')$
are equivalent, and we have $M' = M \cdot P$
by Proposition \ref{Proposition:Key}.
\end{proof}

\begin{cor}
\label{Corollary:NonEmpty}
The set $(k^2 \otimes \Sym_2 k^{n+1})_S$ is non-empty.
(In other words, there is a pair of symmetric matrices whose
discriminant subscheme is $S$.)
\end{cor}

\begin{proof}
Let $\mathcal{F}$ be a free $\O_S$-module of rank one, and take a symmetric isomorphism
$\lambda \colon \mathcal{F} \overset{\sim}{\longrightarrow} \mathcal{F}^{\vee}$
by Lemma \ref{Lemma:Duality:SelfDual}.
By Theorem \ref{Theorem:MainTheorem},
there is a $\GL_{n+1}(k)$-orbit in $(k^2 \otimes \Sym_2 k^{n+1})_S$
corresponding to $(\mathcal{F},\lambda)$.
\end{proof}

\begin{rem}
\label{Remark:NonEmpty}
The above proof shows that, for any coherent $\O_S$-module $\mathcal{F}$ with
$\length_{\O_{S,x}} \O_{S,x} = \length_{\O_{S,x}} {\mathcal{F}}_{x}$
for each $x \in S$,
there is a pair $(\mathcal{F},\lambda)$ corresponding to
a $\GL_{n+1}(k)$-orbit of $(k^2 \otimes \Sym_2 k^{n+1})_S$.
It is also possible to prove Corollary \ref{Corollary:NonEmpty} by a direct calculation
(see Lemma \ref{Appendix:Lemma:LinearAlgebra2}).
\end{rem}

\begin{rem}
By Corollary \ref{Corollary:NonEmpty},
for any non-zero homogeneous polynomial $f(X_0,X_1) \in k[X_0,X_1]$ of degree $n+1$,
we can write
$\det(X_0 M_0 + X_1 M_1) = \lambda f(X_0,X_1)$
for a pair $(M_0,M_1)$ of symmetric matrices of size $n+1$
with entries in $k$ and $\lambda \in k^{\times}$.
It is an interesting problem in Arithmetic Invariant Theory to determine
when we can take $\lambda = 1$.
When $n$ is odd or $k$ is algebraically closed,
replacing $M$ by $M \cdot P$ for some $P \in \GL_{n+1}(k)$,
we can always take $\lambda = 1$.
However, it turns out that,
when $n$ is even and $k$ is not algebraically closed,
we cannot take $\lambda = 1$ in general.
This problem is studied in Section \ref{Section:WoodBhargavaGrossWang}.
\end{rem}

When $k$ is algebraically closed of characteristic different from two,
we prove the following corollary,
which is a simplified version of Theorem \ref{Theorem:MainTheorem}.
Results essentially equivalent to it have already appeared many times in the literature
(e.g.\ \cite[Chapter XIII]{HodgePedoe}).

\begin{cor}
\label{Corollary:MainTheorem:AlgebraicallyClosed}
Assume that $k$ is algebraically closed of characteristic different from two.
There is a bijection between the following sets.
\begin{itemize}
\item The set of $\GL_{n+1}(k)$-orbits in $(k^2 \otimes \Sym_2 k^{n+1})_S$.
\item The set of isomorphism classes of coherent $\O_S$-modules $\mathcal{F}$
with $\length_{\O_{S,x}} \O_{S,x} = \length_{\O_{S,x}} {\mathcal{F}}_{x}$
for each $x \in S$.
\end{itemize}
\end{cor}

\begin{proof}
By Lemma \ref{Lemma:Duality:SelfDual} and Lemma \ref{Lemma:ClassificationQuadraticForm},
for any $\mathcal{F}$ satisfying
the conditions of this corollary,
there is a unique equivalence class of pairs of the form $(\mathcal{F},\lambda)$.
Hence the assertion follows from Theorem \ref{Theorem:MainTheorem}.
\end{proof}

It is straightforward to deduce the following bijection
for coherent $\O_{\P^1}$-modules with zero-dimensional supports
from Corollary \ref{Corollary:MainTheorem:AlgebraicallyClosed}.

\begin{cor}
Assume that $k$ is algebraically closed of characteristic different from two.
There is a bijection between the following sets.
\begin{itemize}
\item The set of $\GL_{n+1}(k)$-orbits in $k^2 \otimes \Sym_2 k^{n+1}$
whose discriminant polynomials are non-zero.
\item The set of isomorphism classes of coherent $\O_{\P^1}$-modules $\mathcal{G}$
with $\dim_k H^0(\P^1,\mathcal{G}) = n+1$ and
$\dim \Supp \mathcal{G} = 0$.
\end{itemize}
\end{cor}

\begin{proof}
For a $\GL_{n+1}(k)$-orbit of $M \in k^2 \otimes \Sym_2 k^{n+1}$ with $\disc(M) \neq 0$,
let $\iota \colon S := D_M \hookrightarrow \P^1$
be the discriminant subscheme.
We take the coherent $\O_S$-module $\mathcal{F}$
corresponding to $M$ by Corollary \ref{Corollary:MainTheorem:AlgebraicallyClosed}.
Then, we put $\mathcal{G} := \iota_{\ast} \mathcal{F}$.
Conversely,
for a coherent $\O_{\P^1}$-module $\mathcal{G}$ with
$\dim_k H^0(\P^1,\mathcal{G}) = n+1$ and $\dim \Supp \mathcal{G} = 0$,
there is a unique closed subscheme $\iota \colon S \hookrightarrow \P^1$
with $\length_{\O_{\P^1,x}} \O_{S,x} = \length_{\O_{\P^1,x}} {\mathcal{G}}_{x}$
for each $x \in S$.
We put $\mathcal{F} := \iota^{\ast} \mathcal{G}$.
Then, we have $\mathcal{G} = \iota_{\ast} \mathcal{F}$.
We take a $\GL_{n+1}(k)$-orbit in $k^2 \otimes \Sym_2 k^{n+1}$
corresponding to $\mathcal{F}$
by Corollary \ref{Corollary:MainTheorem:AlgebraicallyClosed}.
\end{proof}

We shall prove a variant of Theorem \ref{Theorem:MainTheorem} for
$\SL_{n+1}(k)$-orbits.
In order to keep track of the determinants,
we need to consider equivalence classes of triples $(\mathcal{F},\lambda,v)$
rather than pairs $(\mathcal{F},\lambda)$.
Since we have already proved the rigidified bijection,
the proof of the following theorem is straightforward.

\begin{thm}
\label{Theorem:MainTheoremSL(n+1)}
There is a bijection between the following sets.
\begin{itemize}
\item The set of $\SL_{n+1}(k)$-orbits in $(k^2 \otimes \Sym_2 k^{n+1})_S$.
\item The set of equivalence classes of triples $(\mathcal{F},\lambda,v)$, where
\begin{itemize}
\item $\mathcal{F}$ is a coherent $\O_S$-module with
$\length_{\O_{S,x}} \O_{S,x} = \length_{\O_{S,x}} {\mathcal{F}}_{x}$
for each $x \in S$,
\item $\lambda \colon \mathcal{F} \overset{\sim}{\longrightarrow} \mathcal{F}^{\vee}$
is a symmetric isomorphism of $\O_{S}$-modules, and
\item $v \in \bigwedge^{n+1} H^0(S,\mathcal{F})$ is a $k$-basis.
(By the first condition, the coherent $\O_{S}$-modules $\mathcal{F}, \O_S$
have the same composition factors up to permutation.
Hence we have $\dim_k H^0(S,\mathcal{F}) = \dim_k H^0(S,\O_S) = n+1$,
and the wedge product $\bigwedge^{n+1} H^0(S,\mathcal{F})$
is a one-dimensional $k$-vector space.)
\end{itemize}
Here two triples $(\mathcal{F},\lambda,v),(\mathcal{F}',\lambda',v')$
are equivalent if there is an isomorphism
$\rho \colon \mathcal{F} \overset{\sim}{\longrightarrow} \mathcal{F}'$
of $\O_{S}$-modules satisfying
$\lambda = (\transp{\rho}) \circ \lambda' \circ \rho$ and
$\big( \bigwedge^{n+1} H^0(\rho) \big)(v) = v'$.
\end{itemize}
\end{thm}

\begin{proof}
For pairs $M,M' \in (k^2 \otimes \Sym_2 k^{n+1})_S$,
let $(\mathcal{F},\lambda,s)$ (resp.\ $(\mathcal{F}',\lambda',s')$)
be a triple corresponding to $M$ (resp.\ $M'$)
by Proposition \ref{Proposition:Key}.
We put
$v := s_0 \wedge s_1 \wedge \cdots \wedge s_n$ and
$v' := s'_0 \wedge s'_1 \wedge \cdots \wedge s'_n$.
We shall show that
$M,M'$ are in the same $\SL_{n+1}(k)$-orbit if and only if
$(\mathcal{F},\lambda,v)$, $(\mathcal{F}',\lambda',v')$ are equivalent.

Assume that $M,M'$ are in the same $\SL_{n+1}(k)$-orbit.
Take $P \in \SL_{n+1}(k)$ with $M' = M \cdot P$.
By Proposition \ref{Proposition:Key}, there is an isomorphism
$\rho \colon \mathcal{F} \overset{\sim}{\longrightarrow} \mathcal{F}'$
with $\lambda = (\transp{\rho}) \circ \lambda' \circ \rho$
and $\rho(s\,\transp{P}^{-1})  = s'$.
Hence we have
\begin{align*}
v' &= \textstyle \big( \bigwedge^{n+1} H^0(\rho) \big) \big( (s\,\transp{P}^{-1})_0 \wedge (s\,\transp{P}^{-1})_1 \wedge \cdots \wedge (s\,\transp{P}^{-1})_n \big) \\
  &= \textstyle \big( \bigwedge^{n+1} H^0(\rho) \big)\big( (\det \transp{P}^{-1})v \big) \\
  &= \textstyle \big( \bigwedge^{n+1} H^0(\rho) \big)(v)
\end{align*}
because $\det \transp{P}^{-1} = 1$.
Hence $(\mathcal{F},\lambda,v)$, $(\mathcal{F}',\lambda',v')$ are equivalent.

Conversely, assume that
$(\mathcal{F},\lambda,v)$, $(\mathcal{F}',\lambda',v')$ are equivalent.
There is
$\rho \colon \mathcal{F} \overset{\sim}{\longrightarrow} \mathcal{F}'$
satisfying
$\lambda = (\transp{\rho}) \circ \lambda' \circ \rho$ and $\big( \bigwedge^{n+1} H^0(\rho) \big)(v) = v'$.
Since $\rho(s),s'$ are ordered $k$-bases of $H^0(S,\mathcal{F}')$,
there is $P \in \GL_{n+1}(k)$ with $\rho(s\,\transp{P}^{-1})  = s'$.
Then $(\mathcal{F},\lambda,s\,\transp{P}^{-1})$, $(\mathcal{F}',\lambda',s')$
are equivalent in the sense of Proposition \ref{Proposition:Key}.
Hence we have $M' = M \cdot P$.
Since $\big( \bigwedge^{n+1} H^0(\rho) \big)(v) = v'$, we have $\det \transp{P}^{-1} = 1$.
Therefore, we conclude that $P \in \SL_{n+1}(k)$,
and $M,M'$ are in the same $\SL_{n+1}(k)$-orbit.
\end{proof}

\section{Parametrizations of $\SL_{n+1}(k)$-orbits of pairs of symmetric matrices
with fixed discriminant polynomials}
\label{Section:WoodBhargavaGrossWang}

As an application of our results,
we give parametrizations of $\SL_{n+1}(k)$-orbits of pairs of symmetric matrices
with fixed discriminant polynomials generalizing some results of Wood and Bhargava-Gross-Wang
(\cite{Wood}, \cite{BhargavaGrossWang1}, \cite{BhargavaGrossWang2}).

We fix a field $k$, an algebraic closure $\overline{k}$ of $k$,
a positive integer $n \geq 1$,
and a closed subscheme $\iota \colon S \hookrightarrow \P^1$
which is finite over $k$ and satisfies $\dim_k H^0(S,\O_S) = n+1$.
We also fix an isomorphism
$c \colon (\iota^{!}\O_{\P^1}(-1))[1] \overset{\sim}{\longrightarrow} \O_S$
of $\O_S$-modules.

For a pair $M \in (k^2 \otimes \Sym_2 k^{n+1})_S$,
we take $(\mathcal{F},\lambda)$ corresponding to $M$ by
Theorem \ref{Theorem:MainTheorem}.
The isomorphism class of $\mathcal{F}$ is determined by
the $\GL_{n+1}(k)$-orbit of $M$.
We call $\mathcal{F}$ the {\it coherent $\O_S$-module associated with $M$}.
Let
\[ (k^2 \otimes \Sym_2 k^{n+1})_{S,\mathrm{fr}} \subset (k^2 \otimes \Sym_2 k^{n+1})_S \]
be the subset consisting of pairs whose associated $\O_S$-modules are free of rank one.
Here ``fr'' stands for ``free''.
The set $(k^2 \otimes \Sym_2 k^{n+1})_{S,\mathrm{fr}}$
is non-empty by Corollary \ref{Corollary:NonEmpty} (see Remark \ref{Remark:NonEmpty}).

\begin{rem}
\label{Remark:RegularOrbit}
When $\chara k \neq 2$, a pair $M$ is in
$(k^2 \otimes \Sym_2 k^{n+1})_{S,\mathrm{fr}}$ if and only if
the pencil of quadrics associated with $M$ is ``regular'' in the sense of \cite{Wang}.
\end{rem}

\begin{lem}
\label{Lemma:FreenessCriterion}
For a pair $M = (M_0,M_1) \in (k^2 \otimes \Sym_2 k^{n+1})_{S}$,
the following are equivalent.
\begin{itemize}
\item $M \in(k^2 \otimes \Sym_2 k^{n+1})_{S,\mathrm{fr}}$
\item $\rank(a M_0 + b M_1) \geq n$ for any $a,b \in \overline{k}$ with $(a,b) \neq (0,0)$.
\end{itemize}
\end{lem}

\begin{proof}
Let $\mathcal{F}$ be the coherent $\O_S$-module associated with $M$.
By \cite[Chapter VII, \S 4.9]{BourbakiAlgebraII},
$\mathcal{F}$ is free of rank one if and only if
$\mathcal{F} \otimes_k \overline{k}$ is free of rank one.
Hence we may assume $k = \overline{k}$.
Since $\length_{\O_{S,x}} \O_{S,x} = \length_{\O_{S,x}} {\mathcal{F}}_{x}$ for each $x \in S$,
$\mathcal{F}$ is free of rank one if and only if
$\dim_{k} (\mathcal{F} \otimes_{\O_S} (\O_S/m_{S,x})) = 1$ for each $x \in S$,
where $m_{S,x} \subset \O_S$ is the ideal sheaf at $x$.
Let $(a,b)$ be the projective coordinate of $x$.
Then, from the exact sequence
\[
\begin{CD}
0 @>>> \O_{\P^1}(-1)^{\oplus (n+1)}
  @>{M}>> \O_{\P^1}^{\oplus (n+1)}
  @>>> \iota_{\ast} \mathcal{F} @>>> 0,
\end{CD}
\]
we see that
$\dim_{k} (\mathcal{F} \otimes_{\O_S} (\O_S/m_{S,x})) = (n+1) - \rank(a M_0 + b M_1)$.
Hence $\mathcal{F}$ is free of rank one
if and only if
$\rank(a M_0 + b M_1) = n$ for any $(a,b) \in S$.
It is equivalent to the second condition because we always have
$\rank(a M_0 + b M_1) = n+1$ for any $(a,b) \notin S$,
and $\rank(a M_0 + b M_1) \leq n$ for any $(a,b) \in S$.
\end{proof}

\begin{lem}
\label{Lemma:DiscriminantReducedness}
The following are equivalent.
\begin{itemize}
\item $(k^2 \otimes \Sym_2 k^{n+1})_{S,\mathrm{fr}} = (k^2 \otimes \Sym_2 k^{n+1})_S$.
\item $S$ is reduced.
\item There is a homogeneous polynomial $f(X_0,X_1) \in k[X_0,X_1]$ of degree $n+1$
without multiple factors over $k$ satisfying $S = (f = 0)$.
\end{itemize}
\end{lem}

\begin{proof}
The equivalence between the second and the third conditions is immediate.
Assume that the first condition is satisfied.
We may assume $S = \Spec k'[T]/(T^m)$ for some finite extension $k'/k$ and $m \geq 1$
(see the proof of Lemma \ref{Lemma:Duality:SelfDual}).
By the first condition and Lemma \ref{Lemma:Duality:SelfDual},
any finitely generated $k'[T]/(T^m)$-module of length $m$ is free of rank one.
Hence we have $m = 1$, and $S$ is reduced.
Conversely, assume that $S$ is reduced.
Then, $S$ is locally isomorphic to $\Spec k'$ for some finite extension $k'/k$.
Since any finitely generated $k'$-module $M$ with $\length_k M = [k':k]$
is free of rank one, the first condition is satisfied.
\end{proof}

\begin{lem}
\label{Lemma:ClassificationSL(n+1)Orbits:Existence}
There is a pair $M \in (k^2 \otimes \Sym_2 k^{n+1})_{S,\mathrm{fr}}$
such that $\disc(M)$ is written as $\disc(M) = X_1^s \cdot g(X_0,X_1)$,
where the coefficient of $X_0^{n+1-s}$ in $g(X_0,X_1)$ is equal to $(-1)^{n(n+1)/2}$.
\end{lem}

\begin{proof}
See Lemma \ref{Appendix:Lemma:LinearAlgebra2}. See also Lemma \ref{Lemma:FreenessCriterion}.
\end{proof}

The following result is a generalization of some results of Wood and Bhargava-Gross-Wang.
Note that the factor $(-1)^{n(n+1)/2}$ does not appear in
\cite{BhargavaGrossWang1}, \cite{BhargavaGrossWang2}
because it is already incorporated in their definition of discriminant polynomials.

\begin{thm}
\label{Theorem:ClassificationSL(n+1)Orbits}
We put $L := H^0(S,\O_S)$, which is a finite $k$-algebra with $\dim_k L = n+1$.
We define a group $G_S$ by
\[ G_S := (k^{\times} \times L^{\times})/\{ (\Norm_{L/k}(\alpha)^{-1}, \alpha^2) \mid \alpha \in L^{\times} \}. \]
We take a homogeneous polynomial $f(X_0,X_1) \in k[X_0,X_1]$ with $S = (f = 0)$.
We write $f(X_0,X_1) = X_1^s \cdot g(X_0,X_1)$ with a polynomial $g(X_0,X_1)$ not divisible by $X_1$.
Let $C_f \in k^{\times}$ be the coefficient of $X_0^{n+1-s} X_1^s$ in $f(X_0,X_1)$.
(If $(1,0) \notin S$, we have $s = 0$.
If $(1,0) \in S$, the integer $s$ is equal to the multiplicity of $(1,0)$ in $S$.)

\begin{enumerate}
\item There is a simply transitive action of $G_S$ on
the set of $\SL_{n+1}(k)$-orbits in $(k^2 \otimes \Sym_2 k^{n+1})_{S,\mathrm{fr}}$.
For $M \in (k^2 \otimes \Sym_2 k^{n+1})_{S,\mathrm{fr}}$ and $(u,\alpha) \in G_S$,
we denote this action by
\[ \big( [M], (u,\alpha) \big) \mapsto [M] \cdot (u,\alpha). \]
Here we denote the $\SL_{n+1}(k)$-orbit containing $M$ by $[M]$.

\item For any $M \in (k^2 \otimes \Sym_2 k^{n+1})_{S,\mathrm{fr}}$
and $(u,\alpha) \in G_S$,
we have
\[ \disc\big( [M] \cdot (u,\alpha) \big) = u^2 \Norm_{L/k}(\alpha) \disc([M]). \]
(Both hand sides are well-defined because
the action of $\SL_{n+1}(k)$ does not change the discriminant polynomials.)

\item There is a bijection between the following sets.
\begin{itemize}
\item The set of $\SL_{n+1}(k)$-orbits in $(k^2 \otimes \Sym_2 k^{n+1})_{S,\mathrm{fr}}$
whose discriminant polynomials are equal to $f(X_0,X_1)$.
\item The quotient set
\[ \{ \, (u,\alpha) \in k^{\times} \times L^{\times}
  \mid C_f = (-1)^{n(n+1)/2} u^{2} \Norm_{L/k}(\alpha) \, \}/\sim, \]
where the equivalence relation $\sim$ is defined as follows:
$(u,\alpha) \sim (u',\alpha')$ if and only if
$(u,\alpha) = (\Norm_{L/k}(\beta)^{-1} u', \beta^2 \alpha')$ for some $\beta \in L^{\times}$.
\end{itemize}

\item There is a pair $M \in (k^2 \otimes \Sym_2 k^{n+1})_{S,\mathrm{fr}}$
with $\disc(M) = f(X_0,X_1)$
if and only if
\[ C_f = (-1)^{n(n+1)/2} u^2 \Norm_{L/k}(\alpha) \]
for some $u \in k^{\times}$ and $\alpha \in L^{\times}$.

\item Let $u \in k^{\times}$ and $\alpha \in L^{\times}$ be
elements satisfying $C_f = (-1)^{n(n+1)/2} u^{2} \Norm_{L/k}(\alpha)$.
Let $M \in (k^2 \otimes \Sym_2 k^{n+1})_{S,\mathrm{fr}}$
be a pair in the $\SL_{n+1}(k)$-orbit corresponding to $(u,\alpha)$ by (3).
Then the stabilizer of $M$ in $\SL_{n+1}(k)$ is isomorphic to
\[ \Res_{L/k}(\mu_2)_{\Norm_{L/k} = 1} := \{\, \beta \in L^{\times} \mid \beta^2 = 1,\ \Norm_{L/k}(\beta) = 1 \,\}. \]
\end{enumerate}
\end{thm}

\begin{proof}
(1) \ Let $(\mathcal{F},\lambda,v)$ be a triple corresponding to 
an $\SL_{n+1}(k)$-orbit in $(k^2 \otimes \Sym_2 k^{n+1})_{S,\mathrm{fr}}$
by Theorem \ref{Theorem:MainTheoremSL(n+1)}.
Since $\mathcal{F}$ is a free $\O_S$-module of rank one, we may assume $\mathcal{F} = \O_S$.
Then, $\lambda \colon L \overset{\sim}{\longrightarrow} L$ is an $L$-automorphism,
and $v \in \bigwedge^{n+1} L$ is a $k$-basis.
For $(u,\alpha) \in G_S$, we define the map
$(\O_S, \lambda, v) \mapsto (\O_S, \alpha^{-1} \lambda, u^{-1} v)$.
(The morphism $\alpha^{-1} \lambda$ is symmetric by Lemma \ref{Lemma:RankOneSymmetric}.)
This map defines an action of $k^{\times} \times L^{\times}$ on the set of
$\SL_{n+1}(k)$-orbits in $(k^2 \otimes \Sym_2 k^{n+1})_{S,\mathrm{fr}}$.
Let $(\O_S,\lambda',v')$ be a triple corresponding to
another $\SL_{n+1}(k)$-orbit in $(k^2 \otimes \Sym_2 k^{n+1})_{S,\mathrm{fr}}$.
Since  $\lambda,\lambda' \colon L \overset{\sim}{\longrightarrow} L$ are $L$-automorphisms,
there is $\alpha \in L^{\times}$ with $\lambda' = \alpha^{-1} \lambda$.
Since $v,v' \in \bigwedge^{n+1} L$ are $k$-bases,
there is $u \in k^{\times}$ with $v' = u^{-1} v$.
Hence this action is transitive.
For $(u,\alpha) \in G_S$, two triples
$(\O_S, \lambda, v), (\O_S, \alpha^{-1} \lambda, u^{-1} v)$ are equivalent if and only if
there is an $L$-automorphism $\rho \colon L \overset{\sim}{\longrightarrow} L$ with
$\lambda = (\transp{\rho}) \circ (\alpha^{-1} \lambda) \circ \rho$
and $\big( \bigwedge^{n+1} H^0(\rho) \big)(v) = u^{-1} v$
(see Theorem \ref{Theorem:MainTheoremSL(n+1)}).
This condition is satisfied if and only if
$\alpha = \rho(1)^2$ and $u^{-1} = \Norm_{L/k}(\rho(1))$.
Hence the action of $G_S$ is simply transitive.

\vspace{0.1in}

\noindent
(2) \ We write $M = (M_0,M_1)$.
We take a pair $M' = (M'_0,M'_1)$ in the $\SL_{n+1}(k)$-orbit $[M] \cdot (u,\alpha)$.
Consider a triple $(\O_S, \lambda, s)$ (resp.\ $(\O_S, \lambda', s')$)
corresponding to $M$ (resp.\ $M'$) by Proposition \ref{Proposition:Key}.
We have $\lambda' = \alpha^{-1} \lambda$.
There is a unique $P \in \GL_{n+1}(k)$ with $s\,\transp{P}^{-1}  = s'$.
Then we have $\det \transp{P}^{-1} = u^{-1} \iff \det P = u$.
Let $Q \in \GL_{n+1}(k)$ be the matrix representing
the multiplication-by-$\alpha$ map on $L$ with respect to the ordered $k$-basis $s$.
It is straightforward to see from the proof of Proposition \ref{Proposition:Key}
that $(M'_0, M'_1) = (\transp{P} M_0 \transp{Q} P,\, \transp{P} M_1 \transp{Q} P)$.
Since $\det \transp{Q} = \det Q = \Norm_{L/k}(\alpha)$, we have
\begin{align*}
 \disc\big( [M] \cdot (u,\alpha) \big) &= \disc(M') \\
  &= (\det P)^2 (\det Q) \det(X_0 M_0 + X_1 M_1) \\
  &= u^2 \Norm_{L/k}(\alpha) \disc([M]).
\end{align*}

\vspace{0.1in}

\noindent
(3) \ By Lemma \ref{Lemma:ClassificationSL(n+1)Orbits:Existence},
there is a pair $\widetilde{M} \in (k^2 \otimes \Sym_2 k^{n+1})_{S,\mathrm{fr}}$
such that the coefficient of $X_0^{n+1-s} X_1^s$ in $\disc(\widetilde{M})$ is $(-1)^{n(n+1)/2}$.
By (2), we see that
$\disc\big( [\widetilde{M}] \cdot (u,\alpha) \big) = f(X_0,X_1)$
if and only if
$C_f = (-1)^{n(n+1)/2} u^2 \Norm_{L/k}(\alpha)$.
Hence the assertion follows from (1).

\vspace{0.1in}

\noindent
(4) \ This is a special case of (3).

\vspace{0.1in}

\noindent
(5) \ Let $(\O_S, \lambda, s)$ be a triple
corresponding to $M$ by Proposition \ref{Proposition:Key}.
A matrix $P \in \SL_{n+1}(k)$ is in the stabilizer of $M$ (i.e.\ $M \cdot P = M$)
if and only if $(\O_S, \lambda, s)$ is equivalent to $(\O_S, \lambda, s\,\transp{P}^{-1})$.
This condition is satisfied if and only if
there is an $L$-automorphism $\rho \colon L \overset{\sim}{\longrightarrow} L$ with
$\lambda = (\transp{\rho}) \circ \lambda \circ \rho$
and $\rho(s) = s\,\transp{P}^{-1}$.
We put $\beta := \rho(1)$.
Then, it is also equivalent to the condition that $\beta^2 = 1$ and
the matrix representing the multiplication-by-$\beta$ map
with respect to the ordered $k$-basis $s$ is equal to $P$.
Since $P \in \SL_{n+1}(k)$, we have $\Norm_{L/k}(\beta) = 1$.
The map $P \mapsto \beta$ gives a desired isomorphism.
\end{proof}

\begin{rem}
Theorem \ref{Theorem:ClassificationSL(n+1)Orbits} was
proved by Wood and Bhargava-Gross-Wang
when $f(X_0,X_1)$ has no multiple factors over $\overline{k}$ and $(1,0) \notin S$
(\cite[Theorem 7]{BhargavaGrossWang1}, \cite[Corollary 18]{BhargavaGrossWang2}).
(The characteristic of the base field is assume to be different from two
in \cite{BhargavaGrossWang1}, \cite{BhargavaGrossWang2},
but this assumption was not used in the proof of their results corresponding to
Theorem \ref{Theorem:ClassificationSL(n+1)Orbits}.)
In the situation of \cite{BhargavaGrossWang1}, \cite{BhargavaGrossWang2},
we have $s = 0$, and the discriminant subscheme $S \subset \P^1$ is \'etale (hence it is reduced).
\end{rem}

\begin{rem}
Theorem \ref{Theorem:ClassificationSL(n+1)Orbits} only concerns
$\SL_{n+1}(k)$-orbits in $(k^2 \otimes \Sym_2 k^{n+1})_{S,\mathrm{fr}}$.
When $S$ is reduced,
Theorem \ref{Theorem:ClassificationSL(n+1)Orbits} gives
a parametrization of $\SL_{n+1}(k)$-orbits in $(k^2 \otimes \Sym_2 k^{n+1})_{S}$
by Lemma \ref{Lemma:DiscriminantReducedness}.
When $S$ is not reduced,
we have $(k^2 \otimes \Sym_2 k^{n+1})_S \neq (k^2 \otimes \Sym_2 k^{n+1})_{S,\mathrm{fr}}$,
and we do not have a reasonably explicit parametrization of
$\SL_{n+1}(k)$-orbits in the complement
$(k^2 \otimes \Sym_2 k^{n+1})_S \backslash (k^2 \otimes \Sym_2 k^{n+1})_{S,\mathrm{fr}}$.
\end{rem}

\begin{rem}
It is interesting to compare our proofs
and the proofs in \cite{BhargavaGrossWang1}, \cite{BhargavaGrossWang2}.
Basic strategies are similar, but there are some technical differences.
The proofs in \cite{BhargavaGrossWang1}, \cite{BhargavaGrossWang2} are
elementary and depend on explicit calculations on bilinear forms.
Although the arguments in \cite{BhargavaGrossWang1}, \cite{BhargavaGrossWang2}
depend on the assumption of the separability of $f(X_0,X_1)$,
it should be possible to modify their arguments
to give a more elementary proof of Theorem \ref{Theorem:ClassificationSL(n+1)Orbits}.
On the other hand, our proofs are of theoretical nature and not elementary
because we essentially use results from modern algebraic geometry.
An advantage of our methods is that they can be generalized to
classify linear orbits of $m$-tuples of symmetric matrices for $m \geq 3$
(\cite{Ishitsuka}).
Note that our proof of Theorem \ref{Theorem:ClassificationSL(n+1)Orbits} (1),(2)
does not use any explicit calculations on bilinear forms.
(Explicit calculations on bilinear forms are used only in
Lemma \ref{Lemma:ClassificationSL(n+1)Orbits:Existence}.)
\end{rem}

\section{Classification of pencils of quadrics in terms of Segre symbols}
\label{Section:SegreClassification}

In order to illustrate our results,
we shall show how to recover the classical classification of
pencils of quadrics in terms of Segre symbols.
Of course, the methods and the results in this section are more or less well-known.
We use the same notations as in Section \ref{Section:GrothendieckDualityResolution}.
Moreover, in this section, we assume that $k$
is {\it algebraically closed of characteristic different from two}.

Let us recall some definitions and results on pencils of quadrics
(for details, see \cite[Chapter XIII]{HodgePedoe}).
We consider the projective space $\P^n := \Proj k[X_0,\ldots,X_n]$ over $k$.
A {\it quadric} in $\P^n$ is a closed hypersurface $Q \subset \P^n$ defined by
$(F = 0)$ for a quadratic form $F$ in $(n+1)$-variables $X_0,\ldots,X_n$.
A {\it pencil of quadrics} is a family of quadrics in $\P^n$
parametrized by $\P^1$ defined by the equation of the form
$(a F + b G = 0)$ for $(a,b) \in \P^1$,
where $F,G$ are quadratic forms linearly independent over $k$.
We identify the pencil $(a F + b G = 0)$ $((a,b) \in \P^1)$
and the $2$-dimensional $k$-vector space $\langle F,G \rangle$ spanned by $F,G$.
Two pencils of quadrics $\langle F,G \rangle,\ \langle F',G' \rangle$
are said to be {\it projectively equivalent} if there is a linear change of variables
which sends $\langle F,G \rangle$ to $\langle F',G' \rangle$.
Since $\chara k \neq 2$, there are unique symmetric matrices
$A = (a_{i,j})_{0 \leq i,j \leq n}$ and $B = (b_{i,j})_{0 \leq i,j \leq n}$
of size $n+1$ satisfying
\[ F = \sum_{0 \leq i,j \leq n} a_{i,j} X_i X_j, \qquad
   G = \sum_{0 \leq i,j \leq n} b_{i,j} X_i X_j. \]
In this case, we say that the pencil $\langle F,G \rangle$ is defined by
the pair $(A,B)$ of symmetric matrices.
For $(a,b) \in \P^1$,
the quadric $(a F + b G = 0)$ is smooth if and only if $\det(a A + b B) \neq 0$.
Hence the pencil $\langle F,G \rangle$ contains a smooth quadric
if and only if $\det(X_0 A + X_1 B)$ does not vanish identically.
For a pencil $\langle F,G \rangle$ (resp.\ $\langle F',G' \rangle$)
defined by $(A,B)$ (resp.\ $(A',B')$),
it is easy to see that
the two pencils $\langle F,G \rangle, \langle F',G' \rangle$ are projectively equivalent
if and only if
\[ (A,B) = \big( q_{0,0} \transp{P} A' P + q_{1,0} \transp{P} B' P,\ q_{0,1} \transp{P} A' P + q_{1,1} \transp{P} B' P \big) \]
for some $P \in \GL_{n+1}(k)$ and $Q = (q_{i,j})_{0 \leq i,j \leq 1} \in \GL_2(k)$.
For a positive integer $r \geq 1$, we define symmetric matrices $\Delta_r$ and $\Lambda_r$ of size $r$ by
\[
\Delta_r :=
\begin{pmatrix}
 & & & & 1 \\
 & & & 1 \\
 & & \mathrm{\reflectbox{$\ddots$}} \\
 & 1 & & \\
1 & &  &
\end{pmatrix},
\qquad 
\Lambda_r :=
\begin{pmatrix}
 & & & & 0 \\
 & & & 0 & 1 \\
 & & \mathrm{\reflectbox{$\ddots$}} & \mathrm{\reflectbox{$\ddots$}} \\
 & 0 & 1 & & \\
0 & 1 & &  &
\end{pmatrix}.
\]
For square matrices $A_0,\ldots,A_s$, the block diagonal matrix with blocks $A_0,\ldots,A_s$
is denoted by $\bigoplus_{i=0}^s A_i := A_0 \oplus \cdots \oplus A_s$.

The following theorem gives a classification of pencils of quadrics
containing a smooth quadric up to projective equivalence.
For a more traditional proof of this theorem, see  \cite[Chapter XIII, \S 10]{HodgePedoe}.

\begin{thm}
\label{Theorem:Segre}
\begin{enumerate}
\item Let $\langle F,G \rangle$ be a pencil of quadrics containing a smooth quadric.
Then there are different $k$-rational points $(u_0,v_1),\ldots,(u_r,v_r) \in \P^1$
and positive integers $h_i, e_{i,j} \geq 1 \ (0 \leq i \leq r,\ 0 \leq j \leq h_i-1)$
such that
\[ e_{i,0} \geq \cdots \geq e_{i,h_{i-1}} \quad (0 \leq i \leq r), \qquad \sum_{i=0}^r \sum_{j=0}^{h_i-1} e_{i,j} = n+1, \]
and the pencil $\langle F,G \rangle$ is projectively equivalent
to the pencil defined by the pair $(M_0,M_1)$, where
\[
  M_0 := \bigoplus_{i=0}^{r} \bigoplus_{j=0}^{h_i-1} \, ( v_i \Delta_{e_{i,j}} + \Lambda_{e_{i,j}}),
\qquad
  M_1 := \bigoplus_{i=0}^{r} \bigoplus_{j=0}^{h_i-1} \, (-u_i) \Delta_{e_{i,j}}.
\]

\item Let $(M_0,M_1)$ (resp.\ $(M'_0,M'_1)$) be a pair of symmetric matrices
corresponding to the tuple $(r,(u_i,v_i),h_i,e_{i,j})$ (resp.\ $(r',(u'_i,v'_i),h'_i,e'_{i,j})$) in (1).
Then, the pencils defined by $(M_0,M_1), (M'_0,M'_1)$ are projectively equivalent
if and only if $r = r'$,
\ $\sigma(u_i,v_i) = (u'_{\tau(i)}, v'_{\tau(i)})$,
\ $h_i = h'_{\tau(i)}$,
\ $e_{i,j} = e'_{\tau(i),j}$
for an automorphism $\sigma \colon \P^1 \overset{\sim}{\longrightarrow} \P^1$ and
a permutation $\tau \colon \{ 0,1,\ldots,r \} \overset{\sim}{\longrightarrow} \{ 0,1,\ldots,r \}$.
\end{enumerate}
\end{thm}

\begin{proof}
Before giving the proof, let us consider the pair
\[ \widetilde{M} := (\, \widetilde{M}_0,\,\widetilde{M}_1\,) := (v \Delta_{e} + \Lambda_{e},\,- u \Delta_{e}) \]
for $(u,v) \in \P^1$ and $e \geq 1$.
Let $\iota \colon S := D_{\widetilde{M}} \hookrightarrow \P^1$
be the discriminant subscheme.
Since
\[ \disc(\widetilde{M}) := \det\big( X_0 \widetilde{M}_0 + X_1 \widetilde{M}_1 \big) = (-1)^{(e-1)e/2} (v X_0 - u X_1)^e, \]
the scheme $S$ has only one $k$-rational point $(u,v) \in \P^1$.
Let $\widetilde{\mathcal{F}}_{(u,v),e}$ be the coherent $\O_S$-module corresponding to $\widetilde{M}$
by Corollary \ref{Corollary:MainTheorem:AlgebraicallyClosed}.
It is free of rank one by Lemma \ref{Lemma:FreenessCriterion}.

\vspace{0.1in}

\noindent
(1) \ Assume that the pencil $\langle F,G \rangle$ is defined by a pair $(A,B)$.
Let $\mathcal{F}$ be the coherent sheaf corresponding to the pair $(A,B)$
by Corollary \ref{Corollary:MainTheorem:AlgebraicallyClosed}.
The direct image $\iota_{\ast} \mathcal{F}$
can be written as a direct sum of coherent $\O_{\P^1}$-modules of
the form $\iota_{\ast} (\widetilde{\mathcal{F}}_{(u,v),e})$
(\cite[Chapter VII, \S 4.9]{BourbakiAlgebraII}).
Hence the assertion (1) follows.

\vspace{0.1in}

\noindent
(2) \ Let $\mathcal{F}$ (resp.\ $\mathcal{F}'$)
be a coherent sheaf corresponding to $(M_0,M_1)$ (resp.\ $(M'_0,M'_1)$)
by Corollary \ref{Corollary:MainTheorem:AlgebraicallyClosed}.
The pencils defined by $(M_0,M_1), (M'_0,M'_1)$ are projectively equivalent
if and only if $\iota_{\ast} \mathcal{F} \cong \sigma^{\ast} (\iota'_{\ast} \mathcal{F}')$
for an automorphism $\sigma \colon \P^1 \overset{\sim}{\longrightarrow} \P^1$.
Since $h_i,e_{i,j}$ (resp.\ $h'_i,e'_{i,j}$)
are uniquely determined by the isomorphism class of $\iota_{\ast} \mathcal{F}$
(resp.\ $\iota'_{\ast} \mathcal{F}'$)
by \cite[Chapter VII, \S 4.9]{BourbakiAlgebraII}.
The assertion (2) follows.
\end{proof}

\begin{rem}
For a pair $(M_0,M_1)$ as in Theorem \ref{Theorem:Segre} (1),
we have
$\det(X_0 M_0 + X_1 M_1) = \pm \prod_{i=0}^{r} \prod_{j=0}^{h_i-1} (v_i X_0 - u_i X_1)^{e_{i,j}}$.
The factors $(v_i X_0 - u_i X_1)^{e_{i,j}}$ are called {\it elementary divisors}.
Traditionally, the integers $e_{i,j}$ are denoted by the following symbol:
\[ [\, (e_{0,0},\ldots,e_{0,h_0-1}),\, \ldots,\, (e_{r,0},\ldots,e_{r,h_r-1}) \,]. \]
This symbol is called the {\it Segre symbol}.
When $h_i = 1$, we omit the parentheses 
and write ``$e_{i,0}$'' instead of ``$(e_{i,0})$''.
For example, when $h_i = 1$ for all $i$,
the Segre symbol $[(e_{0,0}),\ldots,(e_{r,0})]$
is also denoted by $[e_{0,0},\ldots,e_{r,0}]$.
Pencils of quadrics with Segre symbol $[e_{0,0},\ldots,e_{r,0}]$ correspond to
$\GL_{n+1}(k)$-orbits in $(k^2 \otimes \Sym_2 k^{n+1})_{S,\mathrm{fr}}$.
(See Section \ref{Section:WoodBhargavaGrossWang}.)
\end{rem}

\begin{rem}
Theorem \ref{Theorem:MainTheorem} and Theorem \ref{Theorem:Segre}
are equivalent when $\chara k \neq 2$.
However, when $\chara k = 2$, Theorem \ref{Theorem:Segre} {\it does not hold} because
Corollary \ref{Corollary:MainTheorem:AlgebraicallyClosed} is false in characteristic two.
On the other hand, Theorem \ref{Theorem:MainTheorem} remains true in arbitrary characteristics.
\end{rem}

\section{Classifications of pairs of square matrices which are not necessarily symmetric}
\label{Section:ClassificationPairs:NonSymmetric}

In this section, we give classifications of linear orbits of
pairs of square matrices which are not necessarily symmetric.
The arguments are basically the same as the case of symmetric matrices.
As in Section \ref{Section:GrothendieckDualityResolution},
we fix a field $k$, a positive integer $n \geq 1$,
and a closed subscheme $\iota \colon S \hookrightarrow \P^1$
which is finite over $k$ and satisfies $\dim_k H^0(S,\O_S) = n+1$.
We also fix an isomorphism
$c \colon (\iota^{!}\O_{\P^1}(-1))[1] \overset{\sim}{\longrightarrow} \O_S$
of $\O_S$-modules.

The set of pairs of square matrices of size $n+1$ with entries in $k$ is denoted by
\[ k^2 \otimes k^{n+1} \otimes k^{n+1} :=
   \big\{ M = (M_0,M_1) \ \big| \  M_i \in \mathrm{Mat}_{n+1}(k) \ (i=0,1) \big\}, \]
where the matrices $M_0,M_1$ are not necessarily symmetric.
The following notations are the same as before.
For a pair
$M = (M_0,M_1) \in k^2 \otimes k^{n+1} \otimes k^{n+1}$,
the {\it discriminant polynomial} is defined by
\[ \disc(M) := \det(X_0 M_0 + X_1 M_1). \]
Assume that $\disc(M) \neq 0$.
The {\it discriminant subscheme} is defined by
\[ D_M := (\disc(M) = 0) \subset \P^1. \]
The right action of $P \in \GL_{n+1}(k)$ on $k^2 \otimes k^{n+1} \otimes k^{n+1}$
is defined by
\[ M \cdot P := (\, \transp{P} M_0 P,\, \transp{P} M_1 P\,). \]
Since $\disc(M \cdot P) = (\det P)^2 \disc(M)$,
the action of $\GL_{n+1}(k)$ (resp.\ $\SL_{n+1}(k)$)
preserves discriminant subschemes (resp.\ discriminant polynomials).
Let
\[ (k^2 \otimes k^{n+1} \otimes k^{n+1})_S \subset k^2 \otimes k^{n+1} \otimes k^{n+1} \]
be the subset consisting of pairs whose discriminant subschemes are $S$.

As in the case of symmetric matrices, we shall first prove a rigidified bijection.
The strategy of the proof of the following proposition is
basically the same as Proposition \ref{Proposition:Key}.
For details, see the proof of Proposition \ref{Proposition:Key}.

\begin{prop}
\label{Proposition:Key:NonSymmetric}
\begin{enumerate}
\item There is a bijection between the following sets.
\begin{itemize}
\item The set $(k^2 \otimes k^{n+1} \otimes k^{n+1})_S$.
\item The set of equivalence classes of triples $(\mathcal{F},s,t)$,
where
\begin{itemize}
\item $\mathcal{F}$ is a coherent $\O_{S}$-module with
$\length_{\O_{S,x}} \O_{S,x} = \length_{\O_{S,x}} {\mathcal{F}}_{x}$
for each $x \in S$,
\item $s = \{ s_0,s_1,\ldots,s_n \}$ is an ordered $k$-basis of $H^0(S,\mathcal{F})$, and
\item $t = \{ t_0,t_1,\ldots,t_n \}$ is an ordered $k$-basis of $H^0(S,\mathcal{F}^{\vee})$.
\end{itemize}
Here two triples $(\mathcal{F},s,t),(\mathcal{F}',s',t')$
are equivalent if there is an isomorphism
$\rho \colon \mathcal{F} \overset{\sim}{\longrightarrow} \mathcal{F}'$
of $\O_{S}$-modules with $\rho(s)  = s'$
and $(\transp{\rho})(t')  = t$, where
$\transp{\rho} \colon (\mathcal{F}')^{\vee} \overset{\sim}{\longrightarrow} \mathcal{F}^{\vee}$
is the transpose of $\rho$.
(From the first condition,
we have $\dim_k H^0(S,\mathcal{F}) = \dim_k H^0(S,\O_S) = n+1$.
Since $\mathcal{F}$ is isomorphic to $\mathcal{F}^{\vee}$
by Lemma \ref{Lemma:Duality:SelfDual},
we have $\dim_k H^0(S,\mathcal{F}^{\vee}) = n+1$.)
\end{itemize}

\item Let $M = (M_0,M_1) \in (k^2 \otimes k^{n+1} \otimes k^{n+1})_S$
be a pair corresponding to a triple $(\mathcal{F},s,t)$.
For an invertible matrix $P \in \GL_{n+1}(k)$,
the pair $M \cdot P$ corresponds to
the triple $(\mathcal{F},\,s\,\transp{P}^{-1},\,t\,\transp{P}^{-1})$ by (1).
(For the definition of the ordered $k$-bases
$s\,\transp{P}^{-1}$ and $t\,\transp{P}^{-1}$,
see Proposition \ref{Proposition:Key} (2).)
\end{enumerate}
\end{prop}

\begin{proof}
(1) \ For a pair $M = (M_0,M_1) \in (k^2 \otimes k^{n+1} \otimes k^{n+1})_S$,
we shall construct a triple $(\mathcal{F},s,t)$ as follows.
We define a coherent $\O_{\P^1}$-module $\mathcal{F}_0$ by
\begin{equation}
\label{Proposition:Key:NonSymmetric:ShortExactSequence1}
\begin{CD}
0 @>>> \O_{\P^1}(-1)^{\oplus (n+1)}
  @>{M}>> \O_{\P^1}^{\oplus (n+1)}
  @>{p}>> \mathcal{F}_0 @>>> 0.
\end{CD}
\end{equation}
We put $\mathcal{F} := \iota^{\ast} \mathcal{F}_0$.
Then we have $\mathcal{F}_0 = \iota_{\ast} \mathcal{F}$.
The equality $\length_{\O_{S,x}} \O_{S,x} = \length_{\O_{S,x}} {\mathcal{F}}_{x}$
for each $x \in S$
is proved by the same way as in Proposition \ref{Proposition:Key}.
We denote the image of the standard $k$-basis of $H^0(\P^1,\O_{\P^1}^{\oplus (n+1)})$
by $s = \{ s_0,s_1,\ldots,s_n \}$.
Applying the left derived functor of
$\mathcal{G} \mapsto \Homsheaf_{\P^1}(\mathcal{G},\O_{\P^1}(-1))$
to (\ref{Proposition:Key:NonSymmetric:ShortExactSequence1}), we get
\begin{equation}
\label{Proposition:Key:NonSymmetric:ShortExactSequence2}
\begin{CD}
0 @>>> \O_{\P^1}(-1)^{\oplus (n+1)}
  @>{\transp{M}}>> \O_{\P^1}^{\oplus (n+1)}
  @>{p^{\vee}}>> \Extsheaf^1_{\P^1}(\iota_{\ast} \mathcal{F},\,\O_{\P^1}(-1)) \cong \iota_{\ast}(\mathcal{F}^{\vee})
 @>>> 0.
\end{CD}
\end{equation}
Here we use Lemma \ref{Lemma:Duality:Ext} and the fixed isomorphism $c$.
The image of the standard $k$-basis of $H^0(\P^1,\O_{\P^1}^{\oplus (n+1)})$
by $p^{\vee}$ is denoted by $t = \{ t_0,t_1,\ldots,t_n \}$.

Conversely, assume that a triple $(\mathcal{F},s,t)$ is given.
We define a surjection
$p \colon \O_{\P^1}^{\oplus (n+1)} \longrightarrow \iota_{\ast} \mathcal{F}$
which sends the standard $k$-basis of $H^0(\P^1,\O_{\P^1}^{\oplus (n+1)})$
to the ordered $k$-basis $s = \{ s_0,s_1,\ldots,s_n \}$ of $H^0(S,\mathcal{F})$.
The kernel $\Ker p$ is isomorphic to $\O_{\P^1}(-1)^{\oplus (n+1)}$
by Lemma \ref{Lemma:Resolution}.
Applying the left derived functor of
$\mathcal{G} \mapsto \Homsheaf_{\P^1}(\mathcal{G},\O_{\P^1}(-1))$
to
\begin{equation}
\label{Proposition:Key:NonSymmetric:ShortExactSequence3}
\begin{CD}
0 @>>> \Ker p
  @>>> \O_{\P^1}^{\oplus (n+1)}
  @>{p}>> \iota_{\ast} \mathcal{F} @>>> 0,
\end{CD}
\end{equation}
we get the following short exact sequence
\begin{equation}
\label{Proposition:Key:NonSymmetric:ShortExactSequence4}
\begin{CD}
0 @>>> \O_{\P^1}(-1)^{\oplus (n+1)}
  @>>> \Homsheaf_{\P^1}(\Ker p,\O_{\P^1}(-1))
  @>{p^{\vee}}>> \iota_{\ast}(\mathcal{F}^{\vee})
 @>>> 0.
\end{CD}
\end{equation}
Since the automorphism group of $\O_{\P^1}^{\oplus (n+1)}$ is
isomorphic to $\GL_{n+1}(k)$,
there is a {\it unique} isomorphism
\[ \Homsheaf_{\P^1}(\Ker p,\O_{\P^1}(-1)) \cong \O_{\P^1}^{\oplus (n+1)} \]
such that the morphism $p^{\vee}$ in
(\ref{Proposition:Key:NonSymmetric:ShortExactSequence4})
sends the standard $k$-basis of 
$H^0(\P^1,\O_{\P^1}^{\oplus (n+1)})$
to the given ordered $k$-basis $t = \{ t_0,t_1,\ldots,t_n \}$ of $H^0(S,\mathcal{F}^{\vee})$.
Using it, we get an isomorphism $\Ker p \cong \O_{\P^1}(-1)^{\oplus (n+1)}$.
Hence we have the short exact sequence
\begin{equation}
\label{Proposition:Key:NonSymmetric:ShortExactSequence5}
\begin{CD}
0 @>>> \Ker p \cong \O_{\P^1}(-1)^{\oplus (n+1)}
  @>>> \O_{\P^1}^{\oplus (n+1)}
  @>{p}>> \iota_{\ast} \mathcal{F} @>>> 0,
\end{CD}
\end{equation}
which gives a pair $M$ of square matrices of size $n+1$.

If two triples $(\mathcal{F},s,t)$, $(\mathcal{F}',s',t')$
are equivalent, the matrices obtained from them are the same.
It is easy to see that two maps
$M \mapsto (\mathcal{F},s,t)$ and $(\mathcal{F},s,t) \mapsto M$
constructed as above give a desired bijection.

\vspace{0.1in}

\noindent
(2) \ Let us consider the following commutative diagram
\begin{equation}
\label{Proposition:Key:NonSymmetric:CommutativeDiagram1}
\begin{CD}
0 @>>> \O_{\P^1}(-1)^{\oplus (n+1)}
  @>{M}>> \O_{\P^1}^{\oplus (n+1)}
  @>{p}>> \iota_{\ast} \mathcal{F} @>>> 0 \\
& & @AA{\cong, P}A @AA{\cong, \transp{P}^{-1}}A @AA{\mathrm{id}}A \\
0 @>>> \O_{\P^1}(-1)^{\oplus (n+1)}
  @>{M \cdot P}>> \O_{\P^1}^{\oplus (n+1)} @>{p'}>> \iota_{\ast} \mathcal{F} @>>> 0,
\end{CD}
\end{equation}
where the morphism $p'$ sends the standard $k$-basis of
$H^0(\P^1,\O_{\P^1}^{\oplus (n+1)})$
to the ordered $k$-basis $s\,\transp{P}^{-1}$ of $H^0(S,\mathcal{F})$.
In order to see the effect on $t$,
we apply the left derived functor of
$\mathcal{G} \mapsto \Homsheaf_{\P^1}(\mathcal{G},\O_{\P^1}(-1))$
to (\ref{Proposition:Key:NonSymmetric:CommutativeDiagram1}),
and get
\[
\begin{CD}
0 @>>> \O_{\P^1}(-1)^{\oplus (n+1)}
  @>{\transp{M}}>> \O_{\P^1}^{\oplus (n+1)}
  @>{p^{\vee}}>> \Extsheaf^1_{\P^1}(\iota_{\ast} \mathcal{F},\,\O_{\P^1}(-1)) \cong \iota_{\ast}(\mathcal{F}^{\vee}) @>>> 0 \\
& & @VV{\cong, P^{-1}}V @VV{\cong, \transp{P}}V @VV{\mathrm{id}}V \\
0 @>>> \O_{\P^1}(-1)^{\oplus (n+1)}
  @>{\transp{M} \cdot P}>> \O_{\P^1}^{\oplus (n+1)} @>{(p')^{\vee}}>> \Extsheaf^1_{\P^1}(\iota_{\ast} \mathcal{F},\,\O_{\P^1}(-1)) \cong \iota_{\ast}(\mathcal{F}^{\vee}) @>>> 0.
\end{CD}
\]
Therefore, the morphism $(p')^{\vee}$ sends the standard $k$-basis of
$H^0(\P^1,\O_{\P^1}^{\oplus (n+1)})$
to the ordered $k$-basis $t\,\transp{P}^{-1}$ of $H^0(S,\mathcal{F}^{\vee})$.
\end{proof}

It is straightforward to prove bijections for
$\GL_{n+1}(k)$-orbits and $\SL_{n+1}(k)$-orbits
using Proposition \ref{Proposition:Key:NonSymmetric}.

\begin{thm}
\label{Theorem:MainTheorem:NonSymmetric}
There is a bijection between the following sets.
\begin{itemize}
\item The set of $\GL_{n+1}(k)$-orbits in $(k^2 \otimes k^{n+1} \otimes k^{n+1})_S$.
\item The set of equivalence classes of pairs $(\mathcal{F},\psi)$,
where
\begin{itemize}
\item $\mathcal{F}$ is a coherent $\O_S$-module with
$\length_{\O_{S,x}} \O_{S,x} = \length_{\O_{S,x}} {\mathcal{F}}_{x}$
for each $x \in S$, and
\item $\psi \colon H^0(S,\mathcal{F}) \overset{\sim}{\longrightarrow} H^0(S,\mathcal{F}^{\vee})$
is an isomorphism of $k$-vector spaces.
\end{itemize}
Here two pairs $(\mathcal{F},\psi),(\mathcal{F}',\psi')$
are equivalent if there is an isomorphism
$\rho \colon \mathcal{F} \overset{\sim}{\longrightarrow} \mathcal{F}'$
of $\O_S$-modules satisfying $\psi = H^0(\transp{\rho}) \circ \psi' \circ H^0(\rho)$,
where
$H^0(\rho) \colon H^0(S,\mathcal{F}) \overset{\sim}{\longrightarrow} H^0(S,\mathcal{F}')$
(resp.\ $H^0(\transp{\rho}) \colon H^0(S,(\mathcal{F}')^{\vee}) \overset{\sim}{\longrightarrow} H^0(S,\mathcal{F}^{\vee})$)
is the isomorphism between global sections induced by $\rho$
(resp.\ $\transp{\rho}$).
\end{itemize}
\end{thm}

\begin{proof}
We only give a brief sketch of the proof because
the proof of this theorem is similar to the proof of Theorem \ref{Theorem:MainTheorem}.
For $M \in (k^2 \otimes k^{n+1} \otimes k^{n+1})_S$,
let $(\mathcal{F},s,t)$ be a triple corresponding to $M$
by Proposition \ref{Proposition:Key:NonSymmetric}.
We define an isomorphism
$\psi \colon H^0(S,\mathcal{F}) \overset{\sim}{\longrightarrow} H^0(S,\mathcal{F}^{\vee})$
of $k$-vector spaces by $\psi(s_i) = t_i$ for each $i$.
Similarly, for another pair $M' \in (k^2 \otimes k^{n+1} \otimes k^{n+1})_S$,
we get a pair $(\mathcal{F}',\psi')$.
Using Proposition \ref{Proposition:Key:NonSymmetric},
it is easy to see that $M,M'$ are in the same $\GL_{n+1}(k)$-orbit
if and only if $(\mathcal{F},\psi), (\mathcal{F}',\psi')$ are equivalent.
Details are left to the reader.
\end{proof}

\begin{thm}
\label{Theorem:MainTheoremSL(n+1):NonSymmetric}
There is a bijection between the following sets.
\begin{itemize}
\item The set of $\SL_{n+1}(k)$-orbits in $(k^2 \otimes k^{n+1} \otimes k^{n+1})_S$.
\item The set of equivalence classes of triples $(\mathcal{F},\psi,v)$,
where
\begin{itemize}
\item $\mathcal{F}$ is a coherent $\O_S$-module with
$\length_{\O_{S,x}} \O_{S,x} = \length_{\O_{S,x}} {\mathcal{F}}_{x}$
for each $x \in S$,
\item $\psi \colon H^0(S,\mathcal{F}) \overset{\sim}{\longrightarrow} H^0(S,\mathcal{F}^{\vee})$
is an isomorphism of $k$-vector spaces, and
\item $v \in \bigwedge^{n+1} H^0(S,\mathcal{F})$ is a $k$-basis.
\end{itemize}
Here two triples $(\mathcal{F},\psi,v),(\mathcal{F}',\psi',v')$
are equivalent if there is an isomorphism
$\rho \colon \mathcal{F} \overset{\sim}{\longrightarrow} \mathcal{F}'$
of $\O_S$-modules satisfying $\psi = H^0(\transp{\rho}) \circ \psi' \circ H^0(\rho)$
and $\big( \bigwedge^{n+1} H^0(\rho) \big)(v) = v'$.
\end{itemize}
\end{thm}

\begin{proof}
We omit the proof because it is similar to
the proof of Theorem \ref{Theorem:MainTheoremSL(n+1)}.
Details are left to the reader.
\end{proof}

\begin{rem}
There is a natural map from
the set of $\GL_{n+1}(k)$-orbits in $(k^2 \otimes \Sym_2 k^{n+1})_S$ to
the set of $\GL_{n+1}(k)$-orbits in $(k^2 \otimes k^{n+1} \otimes k^{n+1})_S$.
Via Theorem \ref{Theorem:MainTheorem} and
Theorem \ref{Theorem:MainTheorem:NonSymmetric},
it corresponds to the map
$(\mathcal{F},\lambda) \mapsto (\mathcal{F},\psi)$,
where $\psi = H^0(\lambda)$ is
the isomorphism of $k$-vector spaces induced by $\lambda$.
Similarly, a natural map from
the set of $\SL_{n+1}(k)$-orbits in $(k^2 \otimes \Sym_2 k^{n+1})_S$ to
the set of $\SL_{n+1}(k)$-orbits in $(k^2 \otimes k^{n+1} \otimes k^{n+1})_S$
corresponds to the map
$(\mathcal{F},\lambda,v) \mapsto (\mathcal{F},\psi,v)$ with $\psi = H^0(\lambda)$
by Theorem \ref{Theorem:MainTheoremSL(n+1)} and
Theorem \ref{Theorem:MainTheoremSL(n+1):NonSymmetric}.
\end{rem}

\appendix

\section{Calculation of the determinants of certain symmetric matrices}
\label{Appendix:SymmetricDeterminant}

In this appendix, we calculate the determinants of certain symmetric matrices,
which are symmetric analogues of companion matrices.
The results in this appendix are presumably well-known.
We include the proofs here for the reader's convenience
because we cannot find them in the literature.

Let $k$ be a field, and $\overline{k}$ an algebraic closure of $k$.
We fix a positive integer $r \geq 1$.

\begin{lem}
\label{Appendix:Lemma:LinearAlgebra1}
Let $f(X) \in k[X]$ be a  polynomial  of degree $r$
such that the coefficient of $X^r$ in $f(X)$ is equal to $(-1)^{(r-1)r/2}$.
Then there are symmetric matrices $M_0,M_1$ of size $r$
with entries in $k$ satisfying the following conditions:
\begin{itemize}
\item $\det(X M_0 + M_1) = f(X)$, and
\item $\rank(a M_0 + M_1) \geq r-1$ for any $a \in \overline{k}$.
\end{itemize}
\end{lem}

\begin{proof}
We write $f(X) = \sum_{i=0}^r c_i X^i$, and put $L := k[X]/(f(X))$.
We denote the image of $X$ in $L$ by $\alpha$.
Then $L = k[\alpha]$ is a $k$-vector space of dimension $r$,
and $\{ 1,\alpha,\ldots,\alpha^{r-1} \}$ is a $k$-basis of $L$.
For each $i$, let
$\theta_i \colon L \longrightarrow k$
be the $k$-linear map defined by
$\theta_i(a_0 + a_1 \alpha + \cdots + a_{r-1} \alpha^{r-1}) = a_i$
for any $a_0,a_1,\ldots,a_{r-1} \in k$.
We define symmetric matrices $M_0, M_1$ by
\[ (M_0)_{i,j} := \theta_{r-1}(\alpha^{i+j}), \qquad
   (M_1)_{i,j} := - \theta_{r-1}(\alpha^{i+j+1}) \]
for $0 \leq i,j \leq r-1$.

We shall show these matrices satisfy the conditions of this lemma.
Since $(M_0)_{i,j} = 0$ for $i+j < r-1$
and $(M_0)_{i,j} = 1$ for $i+j = r-1$,
we have $\det M_0 = (-1)^{(r-1)r/2} \neq 0$.
Hence we have
\[ \det\big( X M_0 + M_1 \big) = (-1)^{(r-1)r/2} \det\big( X I_r + (M_0)^{-1} M_1 \big), \]
where $I_r$ is the identity matrix of size $r$.
It is enough to show that
$-(M_0)^{-1} M_1$ is equal to the companion matrix of
the monic polynomial $(-1)^{(r-1)r/2} f(X)$.
Namely, we shall show
\begin{equation}
\label{Appendix:CompanionMatrix}
-(M_0)^{-1} M_1 =
\begin{pmatrix}
0 & 0 & \cdots & 0 & (-1)^{(r-1)r/2+1} c_0 \\
1 & 0 & \cdots & 0 & (-1)^{(r-1)r/2+1} c_1 \\
0 & 1 & \cdots & 0 & (-1)^{(r-1)r/2+1} c_2 \\
0 & 0 & \ddots & 0 & \vdots \\
0 & 0 & \cdots & 1 & (-1)^{(r-1)r/2+1} c_{r-1}
\end{pmatrix}.
\end{equation}

We consider the $k$-vector space $\Hom_k(L,k)$ of $k$-linear maps from $L$ to $k$.
The set $\{ \theta_0,\theta_1,\ldots,\theta_{r-1} \}$ is a $k$-basis of $\Hom_k(L,k)$.
For $\varphi \in \Hom_k(L,k)$ and $x \in L$,
we define $x \varphi \in \Hom_k(L,k)$ by $(x \varphi)(y) := \varphi(x y)$ for $y \in L$.
For $p=0$ or $1$, let
\[ \psi_p \colon \Hom_k(L,k) \longrightarrow \Hom_k(L,k) \]
be the $k$-linear map defined by
$\psi_p(\theta_i) = \alpha^{i+p} \theta_{r-1}$
for $0 \leq i \leq r-1$.
For $x = \sum_{j=0}^{r-1} a_j \alpha^j \in L$ with $a_j \in k$,
we have
\[ \big( \psi_p(\theta_i) \big)(x) = (\alpha^{i+p} \theta_{r-1})(x)
  = \theta_{r-1} \bigg( \sum_{j=0}^{r-1} a_j \alpha^{i+j+p} \bigg)
  = \sum_{j=0}^{r-1} \theta_{r-1}(\alpha^{i+j+p}) \, \theta_j(x). \]
Hence we have
$\psi_p(\theta_i) = \sum_{j=0}^{r-1} \theta_{r-1}(\alpha^{i+j+p}) \, \theta_j$.
For $p = 0$ (resp.\ $p=1$),
the matrix representing $\psi_p$
with respect to $\{ \theta_0,\theta_1,\ldots,\theta_{r-1} \}$
is equal to $M_0$ (resp.\ $-M_1$).
Since $M_0$ is invertible,
$\{ \theta_{r-1},\alpha \theta_{r-1},\ldots,\alpha^{r-1} \theta_{r-1} \}$
is also a $k$-basis of $\Hom_k(L,k)$.
The matrix representing $\psi_1 \circ \psi_0^{-1}$
with respect to
$\{ \theta_{r-1},\alpha \theta_{r-1},\ldots,\alpha^{r-1} \theta_{r-1} \}$
is equal to $- (M_0)^{-1} M_1$.

On the other hand,
$\psi_1 \circ \psi_0^{-1}$ is equal to the map $\varphi \mapsto \alpha \varphi$
because $(\psi_1 \circ \psi_0^{-1})(\alpha^{i} \theta_{r-1}) = \alpha^{i+1} \theta_{r-1}$ for
$0 \leq i \leq r-1$.
Since
\[ \alpha^{r} \theta_{r-1} = (-1)^{(r-1)r/2+1} (\sum_{i=0}^{r-1} c_i \alpha^i) \, \theta_{r-1}, \]
the matrix representing the map $\varphi \mapsto \alpha \varphi$
with respect to
$\{ \theta_{r-1},\alpha \theta_{r-1},\ldots,\alpha^{r-1} \theta_{r-1} \}$
is equal to the right hand side of (\ref{Appendix:CompanionMatrix}).
\end{proof}

\begin{ex}
The entries of $M_0,M_1$ in the proof
can be calculated as polynomials of the coefficients of $f(X)$.
For example, when $r = 4$, the matrices $M_0,M_1$
for $f(X) = X^4 + a X^3 + b X^2 + c X + d$ are
\[
M_0 = \begin{pmatrix}
0 & 0 & 0 & 1 \\
0 & 0 & 1 & -a \\
0 & 1 & -a & a^2-b \\
1 & -a & a^2-b & -a^3+2ab-c
\end{pmatrix},
\]
\[
M_1 = \begin{pmatrix}
0 & 0 & -1 & a \\
0 & -1 & a & -a^2+b \\
-1 & a & -a^2+b & a^3-2ab+c \\
a & -a^2+b & a^3-2ab+c & -a^4 + 3 a^2 b - 2ac -b^2 + d
\end{pmatrix}.
\]
It is easy to confirm that
$- (M_0)^{-1} M_1$ is the companion matrix of $f(X)$.
Hence we have $\det(X M_0 + M_1) = f(X)$.
In fact, we have
\[
- (M_0)^{-1} M_1 =
\begin{pmatrix}
 0 &  0 &  0 & -d \\
 1 &  0 &  0 & -c \\
 0 &  1 &  0 & -b \\
 0 &  0 &  1 & -a
\end{pmatrix}.
\]
\end{ex}

The following result is a variant of Lemma \ref{Appendix:Lemma:LinearAlgebra1}
for homogeneous polynomials.

\begin{lem}
\label{Appendix:Lemma:LinearAlgebra2}
Let $f(X_0,X_1) = \sum_{i=0}^{r} c_i X_0^i X_1^{r-i} \in k[X_0,X_1]$
be a homogeneous polynomial of degree $r$.
Assume that $f(X_0,X_1)$ is written as
$f(X_0,X_1) = X_1^s \cdot g(X_0,X_1)$,
where the coefficient of $X_0^{r-s}$ in $g(X_0,X_1)$ is
equal to $(-1)^{(r-1)r/2}$.
Then there are symmetric matrices $M_0, M_1$ of size $r$
with entries in $k$ satisfying the following conditions:
\begin{itemize}
\item $\det(X_0 M_0 + X_1 M_1) = f(X_0,X_1)$, and
\item $\rank(a M_0 + b M_1) \geq r-1$ for any $a,b \in \overline{k}$ with $(a,b) \neq (0,0)$.
\end{itemize}
\end{lem}

\begin{proof}
Applying Lemma \ref{Appendix:Lemma:LinearAlgebra1} to $g(X,1)$,
we obtain symmetric matrices $M_0, M_1$ of size $r-s$
satisfying
\[ \det(X M_0 + M_1) = (-1)^{(r-s-1)(r-s)/2 + (r-1)r/2} g(X,1) \]
and
$\rank(a M_0 + M_1) \geq r-s-1$ for any $a \in \overline{k}$.
Consider the following matrices $N_0, N_1$ of size $s$
\[
N_0 :=
\begin{pmatrix}
 & & & & 0 \\
 & & & 0 & 1 \\
 & & \mathrm{\reflectbox{$\ddots$}} & \mathrm{\reflectbox{$\ddots$}} \\
 & 0 & 1 & \\
0 & 1 & &
\end{pmatrix},
\qquad
N_1 := (-1)^{r-s}
\begin{pmatrix}
 & & & & 1 \\
 & & & 1 \\
 & & \mathrm{\reflectbox{$\ddots$}} \\
 & 1 & & \\
1 & &  &
\end{pmatrix}.
\]
We take the matrix direct sums
$\widetilde{M}_0 := N_0 \oplus M_0$ and $\widetilde{M}_1 := N_1 \oplus M_1$.
Then $\widetilde{M}_0, \widetilde{M}_1$ satisfy the conditions of this lemma.
In fact, the first condition is satisfied because
\begin{align*}
\det\big( X_0 \widetilde{M}_0 + X_1 \widetilde{M}_1 \big)
  &= \det(X_0 N_0 + X_1 N_1) \cdot \det(X_0 M_0 + X_1 M_1) \\
  &= (-1)^{s(r-s) + (s-1)s/2} X_1^{s} \cdot (-1)^{(r-s-1)(r-s)/2 + (r-1)r/2} g(X_0,X_1) \\
  &= f(X_0,X_1).
\end{align*}
When $b \neq 0$, the second condition is satisfied because $\rank(a N_0 + b N_1) = s$
and $\rank(a M_0 + b M_1) = \rank(a b^{-1} M_0 + M_1) \geq r-s-1$.
When $b = 0$, the second condition is satisfied also in this case
because $\rank (a N_0) = \max \{ s-1,\,0 \}$ and $\rank (a M_0) = r-s$
for any $a \neq 0 \in \overline{k}$.
(The matrix $M_0$ is invertible because
the coefficient of $X^{r-s}$ in $\det(X M_0 + M_1)$ is non-zero.)
\end{proof}

\begin{rem}
It is an interesting problem in Arithmetic Invariant Theory to determine when,
for a given $f(X_0,X_1)$ of degree $r$,
there exist symmetric matrices $M_0, M_1$ of size $r$
with entries in $k$ satisfying
$\det(X_0 M_0 + X_1 M_1) = f(X_0,X_1)$.
Such pairs $M_0, M_1$ do not exist in general
when $k$ is not algebraically closed.
This problem is studied in detail in Section \ref{Section:WoodBhargavaGrossWang}.
\end{rem}

\end{document}